\documentclass[a4paper,12pt]{article}

\usepackage{amsmath,amssymb,amsthm,amscd}
\usepackage{color}

\allowdisplaybreaks

\newcommand{\Fg}{\mathfrak{g}}
\newcommand{\Fh}{\mathfrak{h}}

\newcommand{\BC}{\mathbb{C}}
\newcommand{\BR}{\mathbb{R}}
\newcommand{\BQ}{\mathbb{Q}}
\newcommand{\BZ}{\mathbb{Z}}
\newcommand{\BB}{\mathbb{B}}
\newcommand{\BP}{\mathbb{P}}

\newcommand{\Hom}{\mathop{\rm Hom}\nolimits}
\newcommand{\GL}{\mathop{\rm GL}\nolimits}
\newcommand{\wt}{\mathop{\rm wt}\nolimits}
\newcommand{\cl}{\mathop{\rm cl}\nolimits}
\newcommand{\dist}{\mathop{\rm dist}\nolimits}
\newcommand{\INT}{\mathop{\rm int}\nolimits}

\newcommand{\rr}{\Delta_{\mathrm{re}}}
\newcommand{\prr}{\Delta_{\mathrm{re}}^{+}}

\newcommand{\pair}[2]{\langle #1,\,#2 \rangle}

\newcommand{\len}[2]{\ell(#2,\,#1)}

\newcommand{\mcr}[1]{\lfloor #1 \rfloor}

\newcommand{\ve}{\varepsilon}
\newcommand{\vp}{\varphi}
\newcommand{\vpi}{\varpi}

\newcommand{\ha}[1]{\widehat{#1}}
\newcommand{\ti}[1]{\widetilde{#1}}

\newcommand{\ud}[1]{\underline{#1}}

\newcommand{\bzero}{{\bf 0}}
\newcommand{\bd}{{\bf d}}

\makeatletter
\@addtoreset{equation}{subsection}

\renewcommand\section{\@startsection{section}{1}{0pt}
{-3.5ex plus -1ex minus -.2ex}{1.0ex plus .2ex}{\large\bf}}
\renewcommand\subsection{\@startsection{subsection}{1}{0pt}
{2.5ex plus 1ex minus .2ex}{-1em}{\bf}}
\makeatother

\theoremstyle{plain}
\newtheorem{thm}{Theorem}[subsection]
\newtheorem{lem}[thm]{Lemma}
\newtheorem{prop}[thm]{Proposition}

\newtheorem*{claim*}{Claim}

\theoremstyle{definition}
\newtheorem{dfn}[thm]{Definition}

\theoremstyle{remark}
\newtheorem{rem}[thm]{Remark}

\newenvironment{enu}{%
 \begin{enumerate}%
 \renewcommand{\labelenumi}{\rm (\theenumi)}%
}{\end{enumerate}}

\begin{document}

\setlength{\baselineskip}{17.9pt}

\title{\Large\bf 
Quantum Lakshmibai-Seshadri paths \\[3mm] and root operators}
\author{
 Cristian Lenart \\ 
 \small Department of Mathematics and Statistics, 
 State University of New York at Albany, \\ 
 \small Albany, NY 12222, U.\,S.\,A. \ 
 (e-mail: {\tt clenart@albany.edu}) \\[5mm]
 Satoshi Naito \\ 
 \small Department of Mathematics, Tokyo Institute of Technology, \\
 \small 2-12-1 Oh-okayama, Meguro-ku, Tokyo 152-8551, Japan \ 
 (e-mail: {\tt naito@math.titech.ac.jp}) \\[5mm]
 Daisuke Sagaki \\ 
 \small Institute of Mathematics, University of Tsukuba, \\
 \small Tsukuba, Ibaraki 305-8571, Japan \ 
 (e-mail: {\tt sagaki@math.tsukuba.ac.jp}) \\[5mm]
 Anne Schilling \\ 
 \small Department of Mathematics, University of California, \\
 \small One Shields Avenue, Davis, CA 95616-8633, U.\,S.\,A. \ 
 (e-mail: {\tt anne@math.ucdavis.edu}) \\[5mm]
 Mark Shimozono \\ 
 \small Department of Mathematics, MC 0151, 460 McBryde Hall, 
        Virginia Tech, \\
 \small 225 Stanger St., Blacksburg, VA 24061, 
 U.\,S.\,A. \ 
 (e-mail: {\tt mshimo@vt.edu})
}
\date{}
\maketitle

%
\begin{abstract} 
We give an explicit description of the image of a quantum LS path, 
regarded as a rational path, under the action of root operators, and 
show that the set of quantum LS paths is stable under the action 
of the root operators. As a by-product, we obtain a new proof 
of the fact that a projected level-zero LS path is 
just a quantum LS path.
\end{abstract}
%
%
\section{Introduction.}
\label{sec:intro}
In our previous papers \cite{NS-IMRN}, \cite{NS-Adv}, \cite{NS-Tensor}, 
we gave a combinatorial realization of the crystal bases of 
level-zero fundamental representations $W(\vpi_{i})$, 
$i \in I_{0}$, and their tensor products 
$\bigotimes_{i \in I_{0}} W(\vpi_{i})^{\otimes m_{i}}$, 
$m_{i} \in \BZ_{\geq 0}$, 
over quantum affine algebras $U_{q}'(\Fg)$, 
by using projected level-zero Lakshmibai-Seshadri (LS for short) paths.
Here, for a level-zero dominant integral weight 
$\lambda = \sum_{i \in I_{0}} m_{i} \vpi_{i}$, 
with $\vpi_{i}$ the $i$-th level-zero fundamental weight, 
the set of projected level-zero LS paths of shape $\lambda$, 
which is a ``simple'' crystal denoted by 
$\BB(\lambda)_{\cl}$, is obtained from 
the set $\BB(\lambda)$ of LS paths of shape $\lambda$ 
(in the sense of \cite{Lit-A}) 
by factoring out the null root $\delta$ 
of an affine Lie algebra $\mathfrak{g}$.
However, from the nature of the above definition of 
projected level-zero LS paths, our description of 
these objects in \cite{NS-IMRN}, \cite{NS-Adv}, \cite{NS-Tensor}
was not as explicit as the one (given in \cite{Lit-I}) 
of usual LS paths, the shape of which 
is a dominant integral weight.

Recently, in \cite{LNSSS1}, \cite{LNSSS2}, 
we proved that a projected level-zero LS path is 
identical to a certain ``rational path'', 
which we call a quantum LS path. 
A quantum LS path is described in terms of 
the (parabolic) quantum Bruhat graph (QBG for short), 
which was introduced by \cite{BFP} 
(and by \cite{LS} in the parabolic case)
in the study of the quantum cohomology ring of 
the (partial) flag variety; see \S\ref{subsec:def-QBG} 
for the definition of the (parabolic) QBG.
It is noteworthy that the description of 
a quantum LS path as a rational path is 
very similar to the one of a usual LS path given in \cite{Lit-I}, 
in which we replace the Hasse diagram of the (parabolic) 
Bruhat graph by the (parabolic) QBG. 
Also, remark that the vertices of the (parabolic) QBG are 
the minimal-length representatives 
for the cosets of a parabolic subgroup $W_{0,\,J}$ 
of the finite Weyl group $W_{0}$, though we consider 
finite-dimensional representations $W(\vpi_{i})$, 
$i \in I_{0}$, of quantum affine algebras $U_{q}'(\Fg)$.

The purpose of this paper is to give 
an explicit description, in terms of rational paths, 
of the image of a quantum LS path ($=$ projected level-zero LS path) 
under root operators in a way similar to the one given 
in \cite{Lit-I}; see Theorem~\ref{thm:main} for details.
This explicit description, together with 
the Diamond Lemmas \cite[Lemma 5.14]{LNSSS1}, 
for the parabolic QBG, provides us with a proof of 
the fact that the set of quantum LS paths 
(the shape of which is a level-zero dominant 
integral weight $\lambda$) is stable under 
the action of the root operators.

As a by-product of the stability property above, 
we obtain another (but somewhat roundabout) proof of 
the fact that a projected level-zero LS path is 
just a quantum LS path; see \cite{LNSSS1}, \cite{LNSSS2} 
for a more direct proof.
This new proof is accomplished by making use of 
a characterization (Theorem~\ref{thm:charls}) of the set 
$\BB(\lambda)_{\cl}$ of 
projected level-zero LS paths of shape $\lambda$ 
in terms of root operators, which is based upon 
the connectedness of the (crystal graph for the) 
tensor product crystal 
$\bigotimes_{i \in I_{0}} 
 \BB(\vpi_{i})_{\cl}^{\otimes m_{i}} 
 \simeq \BB(\lambda)_{\cl}$; 
recall from \cite{NS-IMRN}, \cite{NS-Adv}, \cite{NS-Tensor} that
for a level-zero dominant integral weight 
$\lambda = \sum_{i \in I_{0}} m_{i} \vpi_{i}$, 
the crystal $\BB(\lambda)_{\cl}$ decomposes into 
the tensor product $\bigotimes_{i \in I_{0}} 
\BB(\vpi_{i})^{\otimes m_{i}}_{\cl}$ of crystals, 
and that $\BB(\vpi_{i})_{\cl}$ for each $i \in I_{0}$ is 
isomorphic to the crystal basis of the level-zero fundamental 
representation $W(\vpi_{i})$.

This paper is organized as follows. 
In \S\ref{sec:LS}, we fix our fundamental notation, and 
recall some basic facts about (level-zero) LS path crystals. 
Also, we give a characterization (Theorem~\ref{thm:charls}) of 
projected level-zero LS paths, which is needed to obtain 
our main result (Theorem~\ref{thm:main}). In \S\ref{sec:QLS}, 
we recall the notion of the (parabolic) quantum Bruhat graph, 
and then give the definition of quantum LS paths. 
In \S\ref{sec:main}, we first state our main result. 
Then, after preparing several technical lemmas, 
we finally obtain an explicit description (Proposition~\ref{prop:stable})
of the image of a quantum LS path as a rational path under the 
action of root operators. 
Our main result follows immediately from this description, together 
with the characterization above of projected level-zero LS paths.

\paragraph{Acknowledgments.}
C.L. was partially supported by the NSF grant DMS--1101264. 
S.N. was supported by Grant-in-Aid for Scientific Research (C), No.\,24540010, Japan. 
D.S. was supported by Grant-in-Aid for Young Scientists (B) No.\,23740003, Japan.
A.S. was partially supported by the NSF grants DMS--1001256, OCI--1147247, 
and a grant from the Simons Foundation (\#226108 to Anne Schilling).
M.S. was partially supported by the NSF grant DMS--1200804.

%
\section{Lakshmibai-Seshadri paths.}
\label{sec:LS}

%
\subsection{Basic notation.}
\label{subsec:notation}

Let $\Fg$ be an untwisted affine Lie algebra 
over $\BC$ with Cartan matrix $A=(a_{ij})_{i,\,j \in I}$; 
throughout this paper, the elements of the index set $I$ 
are numbered as in \cite[\S4.8, Table Aff~1]{Kac}. 
Take a distinguished vertex $0 \in I$ as in \cite{Kac}, 
and set $I_{0}:=I \setminus \{0\}$.
Let 
$\Fh=\bigl(\bigoplus_{j \in I} \BC \alpha_{j}^{\vee}\bigr) \oplus \BC d$
denote the Cartan subalgebra of $\Fg$, where 
$\Pi^{\vee}:=\bigl\{\alpha_{j}^{\vee}\bigr\}_{j \in I} \subset \Fh$ is 
the set of simple coroots, and 
$d \in \Fh$ is the scaling element (or degree operator). 
Also, we denote by 
$\Pi:=\bigl\{\alpha_{j}\bigr\}_{j \in I} \subset 
\Fh^{\ast}:=\Hom_{\BC}(\Fh,\BC)$ 
the set of simple roots, and by 
$\Lambda_{j} \in \Fh^{\ast}$, $j \in I$, 
the fundamental weights; 
note that $\alpha_{j}(d)=\delta_{j,0}$ and 
$\Lambda_{j}(d)=0$ for $j \in I$. 
Let $\delta=\sum_{j \in I} a_{j}\alpha_{j} \in \Fh^{\ast}$ and 
$c=\sum_{j \in I} a^{\vee}_{j} \alpha_{j}^{\vee} \in \Fh$ denote 
the null root and the canonical central element of 
$\Fg$, respectively. 
The Weyl group $W$ of $\Fg$ is defined by 
$W:=\langle r_{j} \mid j \in I\rangle \subset \GL(\Fh^{\ast})$, 
where $r_{j} \in \GL(\Fh^{\ast})$ denotes the simple reflection 
associated to $\alpha_{j}$ for $j \in I$, with
$\ell:W \rightarrow \BZ_{\ge 0}$ the length function on $W$. 
Denote by $\rr$ the set of real roots, i.e., $\rr:=W\Pi$, 
and by $\prr \subset \rr$ the set of positive real roots; 
for $\beta \in \rr$, we denote by $\beta^{\vee}$ 
the dual root of $\beta$, and by $r_{\beta} \in W$ 
the reflection with respect to $\beta$. 
We take a dual weight lattice $P^{\vee}$ 
and a weight lattice $P$ as follows:
%
%
\begin{equation} \label{eq:lattices}
P^{\vee}=
\left(\bigoplus_{j \in I} \BZ \alpha_{j}^{\vee}\right) \oplus \BZ d \, 
\subset \Fh
\quad \text{and} \quad 
P= 
\left(\bigoplus_{j \in I} \BZ \Lambda_{j}\right) \oplus 
   \BZ \delta \subset \Fh^{\ast}.
\end{equation}
It is clear that $P$ contains 
$Q:=\bigoplus_{j \in I} \BZ \alpha_{j}$, and that 
$P \cong \Hom_{\BZ}(P^{\vee},\BZ)$. 

Let $W_{0}$ be the subgroup of $W$ generated by 
$r_{j}$, $j \in I_{0}$, and set $\Delta_{0}:=
\rr \cap \bigoplus_{j \in I_{0}}\BZ\alpha_{j}$, 
$\Delta_{0}^{+}:=
\rr \cap \bigoplus_{j \in I_{0}}\BZ_{\ge 0}\alpha_{j}$, and
$\Delta_{0}^{-}:=-\Delta_{0}^{+}$. 
Note that $W_{0}$ (resp., $\Delta_{0}$, $\Delta_{0}^{+}$, $\Delta_{0}^{-}$) 
can be thought of as the (finite) Weyl group 
(resp., the set of roots, the set of positive roots, the set of negative roots) 
of the finite-dimensional simple Lie algebra corresponding to $I_{0}$. 
Denote by $\theta \in \Delta_{0}^{+}$ the highest root for the 
(finite) root system $\Delta_{0}$; note that 
$\alpha_{0}=-\theta+\delta$ and $\alpha_{0}^{\vee}=-\theta^{\vee}+c$. 
%
%
%
\begin{dfn} \label{dfn:lv0}
\mbox{}
\begin{enu}
\item
An integral weight 
$\lambda \in P$ is said to 
be of level zero if $\pair{\lambda}{c}=0$. 

\item
An integral weight 
$\lambda \in P$ is said to 
be level-zero dominant if $\pair{\lambda}{c}=0$, and 
$\pair{\lambda}{\alpha_{j}^{\vee}} \ge 0$ for all $j \in I_{0}=I \setminus \{0\}$. 
\end{enu}
\end{dfn}
%
%
\begin{rem} \label{rem:theta}
If $\lambda \in P$ is of level zero, 
then $\pair{\lambda}{\alpha_{0}^{\vee}}=-\pair{\lambda}{\theta^{\vee}}$.
\end{rem}

For each $i \in I_{0}$, 
we define a level-zero fundamental weight 
$\vpi_{i} \in P$ by
\begin{equation}
\vpi_{i}:=\Lambda_{i}-a_{i}^{\vee}\Lambda_{0}.
\end{equation}
The $\vpi_{i}$ for $i \in I_{0}$ is actually
a level-zero dominant integral weight; 
indeed, $\pair{\vpi_{i}}{c}=0$ and 
$\pair{\vpi_{i}}{\alpha_{j}^{\vee}}=\delta_{i,j}$ for $j \in I_{0}$.

Let 
$\cl:\Fh^{\ast} \twoheadrightarrow \Fh^{\ast}/\BC\delta$ 
be the canonical projection from $\Fh^{\ast}$ onto 
$\Fh^{\ast}/\BC\delta$, and 
define $P_{\cl}$ and $P_{\cl}^{\vee}$ by
%
%
\begin{equation} \label{eq:lat-cl}
P_{\cl} := \cl(P) = 
 \bigoplus_{j \in I} \BZ \cl(\Lambda_{j})
\quad \text{and} \quad 
P_{\cl}^{\vee} := 
 \bigoplus_{j \in I} \BZ \alpha_{j}^{\vee} 
 \subset P^{\vee}.
\end{equation}
We see that 
$P_{\cl} \cong P/\BZ\delta$, and that 
$P_{\cl}$ can be identified with 
$\Hom_{\BZ}(P_{\cl}^{\vee},\BZ)$ as a $\BZ$-module by
%
%
\begin{equation} \label{eq:pair}
\pair{\cl(\lambda)}{h}=\pair{\lambda}{h} \quad
\text{for $\lambda \in P$ and $h \in P_{\cl}^{\vee}$}.
\end{equation}
Also, there exists 
a natural action of the Weyl group $W$ on 
$\Fh^{\ast}/\BC\delta$ induced by the one on $\Fh^{\ast}$, 
since $W\delta=\delta$; it is obvious that
$w \circ \cl = \cl \circ w$ for all $w \in W$.
%
%
\begin{rem} \label{rem:orbcl}
Let $\lambda \in P$ be a level-zero integral weight. 
It is easy to check that 
$\cl(W\lambda)=W_{0}\cl(\lambda)$ 
(see the proof of \cite[Lemma~2.3.3]{NS-LMS}). 
In particular, we have 
$\cl(r_{0}\lambda)=\cl(r_{\theta}\lambda)$ 
since $\alpha_{0}=-\theta+\delta$ and 
$\alpha_{0}^{\vee}=-\theta^{\vee}+c$. 
\end{rem}

For simplicity of notation, 
we often write $\beta$ instead of $\cl(\beta) \in P_{\cl}$ 
for $\beta \in Q=\bigoplus_{j \in I} \BZ \alpha_{j}$; 
note that $\alpha_{0}=-\theta$ in $P_{\cl}$ 
since $\alpha_{0}=-\theta+\delta$ in $P$. 

%
\subsection{Paths and root operators.}
\label{subsec:path}

A path with weight in $P_{\cl}=\cl(P)$ is, by definition, 
a piecewise-linear, continuous map 
$\pi:[0,1] \rightarrow \BR \otimes_{\BZ} P_{\cl}$
such that $\pi(0)=0$ and $\pi(1) \in P_{\cl}$. 
We denote by $\BP_{\cl}$ the set of all paths with weight in $P_{\cl}$, 
and define $\wt:\BP_{\cl} \rightarrow P_{\cl}$ by
%
%
\begin{equation} \label{eq:wt}
\wt(\eta):=\eta(1) \quad 
 \text{\rm for $\eta \in \BP_{\cl}$}. 
\end{equation}
For $\eta \in \BP_{\cl}$ and $j \in I$, we set
%
%
\begin{equation} \label{eq:Hm}
\begin{array}{l}
H^{\eta}_{j}(t):=\pair{\eta(t)}{\alpha_{j}^{\vee}} 
  \quad \text{for \,} t \in [0,1], \\[3mm]
m^{\eta}_{j}
  :=\min\bigl\{H^{\eta}_{j}(t) \mid t \in [0,1]\bigr\}. 
\end{array}
\end{equation}
For each $j \in I$, 
let $\BP_{\cl,\,\INT}^{(j)}$ denote the subset of $\BP_{\cl}$ 
consisting of all paths $\eta$ for which 
all local minima of the function
$H^{\eta}_{j}(t)$ are integers; 
note that if $\eta \in \BP_{\cl,\,\INT}^{(j)}$, 
then $m^{\eta}_{j} \in \BZ_{\le 0}$ and 
$H^{\eta}_{j}(1)-m^{\eta}_{j} \in \BZ_{\ge 0}$. 
We set 
\begin{equation*}
\BP_{\cl,\,\INT}:=\bigcap_{j \in I} \BP_{\cl,\,\INT}^{(j)};
\end{equation*}
see also \cite[\S2.3]{NS-Tensor}.
Here we should warn the reader that 
the set $\BP_{\cl,\,\INT}$ itself is not necessarily 
stable under the action of the root operators $e_{j}$ and $f_{j}$ 
for $j \in I$, defined below.

Now, for $j \in I$ and $\eta \in \BP_{\cl,\,\INT}^{(j)}$, 
we define $e_{j}\eta$ as follows. 
If $m^{\eta}_{j} = 0$, then $e_{j}\eta:=\bzero$, 
where $\bzero$ is an additional element 
not contained in $\BP_{\cl}$. 
If $m^{\eta}_{j} \le -1$, 
then we define $e_{j}\eta \in \BP_{\cl}$ by
%
%
\begin{equation} \label{eq:ro_e}
(e_{j}\eta)(t):=
\begin{cases}
\eta(t) & \text{if \,} 0 \le t \le t_{0}, \\[2mm]
\eta(t_{0})+r_{j}(\eta(t)-\eta(t_{0}))
       & \text{if \,} t_{0} \le t \le t_{1}, \\[2mm]
\eta(t)+\alpha_{j} & \text{if \,} t_{1} \le t \le 1,
\end{cases}
\end{equation}
where we set 
%
%
\begin{equation} \label{eq:t1}
\begin{array}{l}
t_{1}:=\min\bigl\{t \in [0,1] \mid 
       H^{\eta}_{j}(t)=m^{\eta}_{j} \bigr\}, \\[2mm]
t_{0}:=\max\bigl\{t \in [0,t_{1}] \mid
       H^{\eta}_{j}(t) = m^{\eta}_{j}+1 \bigr\};
\end{array}
\end{equation}
note that the function $H^{\eta}_{j}(t)$ is strictly decreasing 
on $[t_{0},t_{1}]$ since $\eta \in \BP_{\cl,\,\INT}^{(j)}$. Because 
\begin{equation*}
H^{e_{j}\eta}_{j}(t)=
\begin{cases}
H^{\eta}_{j}(t) & \text{if \,} 0 \le t \le t_{0}, \\[2mm]
2(m^{\eta}_{j}+1)-H^{\eta}_{j}(t)
       & \text{if \,} t_{0} \le t \le t_{1}, \\[2mm]
H^{\eta}_{j}(t)+2 & \text{if \,} t_{1} \le t \le 1,
\end{cases}
\end{equation*}
it is easily seen that 
$e_{j}\eta \in \BP_{\cl,\,\INT}^{(j)}$, and 
$m^{e_{j}\eta}_{j}=m^{\eta}_{j}+1$. Therefore, if we set
%
%
\begin{equation} \label{eq:ve}
\ve_{j}(\eta):=
 \max\bigl\{n \ge 0 \mid e_{j}^{n}\eta \ne \bzero\bigr\}
\end{equation}
for $j \in I$ and $\eta \in \BP_{\cl,\,\INT}^{(j)}$,
then $\ve_{j}(\eta)=-m^{\eta}_{j}$ 
(see also \cite[Lemma 2.1\,c)]{Lit-A}). By convention, 
we set $e_{j}\bzero:=\bzero$ for all $j \in I$. 

\begin{rem}
Assume that $\eta \in \BP_{\cl,\,\INT}^{(0)}$ satisfies the condition that 
$m^{\eta}_{0} \le -1$ and $\pair{\eta(t)}{c}=0$ for all $t \in [0,1]$. 
Then we have
%
%
\begin{equation} \label{eq:ro_e0}
(e_{0}\eta)(t)=
\begin{cases}
\eta(t) & \text{if \,} 0 \le t \le t_{0}, \\[2mm]
\eta(t_{0})+r_{\theta}(\eta(t)-\eta(t_{0}))
       & \text{if \,} t_{0} \le t \le t_{1}, \\[2mm]
\eta(t)-\theta & \text{if \,} t_{1} \le t \le 1,
\end{cases}
\end{equation}
where $t_{0}$ and $t_{1}$ are defined by \eqref{eq:t1} 
for $j=0$. 
\end{rem}

Similarly, for $j \in I$ and 
$\eta \in \BP_{\cl,\,\INT}^{(j)}$,
we define $f_{j}\eta$ as follows. 
If $H^{\eta}_{j}(1)-m^{\eta}_{j}=0$, 
then $f_{j}\eta:=\bzero$. 
If $H^{\eta}_{j}(1)-m^{\eta}_{j} \ge 1$, 
then we define $f_{j}\eta \in \BP_{\cl}$ by
%
%
\begin{equation} \label{eq:ro_f}
(f_{j}\eta)(t):=
\begin{cases}
\eta(t) & \text{if \,} 0 \le t \le t_{0}, \\[2mm]
\eta(t_{0})+r_{j}(\eta(t)-\eta(t_{0}))
       & \text{if \,} t_{0} \le t \le t_{1}, \\[2mm]
\eta(t)-\alpha_{j} & \text{if \,} t_{1} \le t \le 1,
\end{cases}
\end{equation}
where we set 
%
%
\begin{equation} \label{eq:t2}
\begin{array}{l}
t_{0}:=\max\bigl\{t \in [0,1] \mid 
       H^{\eta}_{j}(t)=m^{\eta}_{j}\bigr\}, \\[2mm]
t_{1}:=\min\bigl\{t \in [t_{0},1] \mid
       H^{\eta}_{j}(t)=m^{\eta}_{j}+1\bigr\}; 
\end{array}
\end{equation}
note that the function $H^{\eta}_{j}(t)$ is strictly increasing 
on $[t_{0},t_{1}]$ since $\eta \in \BP_{\cl,\,\INT}^{(j)}$. 
Because 
\begin{equation*}
H^{f_{j}\eta}_{j}(t)=
\begin{cases}
H^{\eta}_{j}(t) & \text{if \,} 0 \le t \le t_{0}, \\[2mm]
2m^{\eta}_{j}-H^{\eta}_{j}(t)
       & \text{if \,} t_{0} \le t \le t_{1}, \\[2mm]
H^{\eta}_{j}(t)-2 & \text{if \,} t_{1} \le t \le 1,
\end{cases}
\end{equation*}
it is easily seen that 
$f_{j}\eta \in \BP_{\cl,\,\INT}^{(j)}$, and 
$m^{f_{j}\eta}_{j}=m^{\eta}_{j}-1$. Therefore, if we set
%
%
\begin{equation} \label{eq:vp}
\vp_{j}(\eta):=
 \max\bigl\{n \ge 0 \mid f_{j}^{n}\eta \ne \bzero\bigr\}
\end{equation}
for $j \in I$ and $\eta \in \BP_{\cl,\,\INT}^{(j)}$, 
then $\vp_{j}(\eta)=H^{\eta}_{j}(1)-m^{\eta}_{j}$ 
(see also \cite[Lemma 2.1\,c)]{Lit-A}).
By convention, we set $f_{j}\bzero:=\bzero$ 
for all $j \in I$. 

\begin{rem} \label{rem:ro_f0}
Assume that $\eta \in \BP_{\cl,\,\INT}^{(0)}$ satisfies the condition that 
$H^{\eta}_{0}(1)-m^{\eta}_{0} \ge 1$ and $\pair{\eta(t)}{c}=0$ for all $t \in [0,1]$. 
Then we have
%
%
\begin{equation} \label{eq:ro_f0}
(f_{0}\eta)(t)=
\begin{cases}
\eta(t) & \text{if \,} 0 \le t \le t_{0}, \\[2mm]
\eta(t_{0})+r_{\theta}(\eta(t)-\eta(t_{0}))
       & \text{if \,} t_{0} \le t \le t_{1}, \\[2mm]
\eta(t)+\theta & \text{if \,} t_{1} \le t \le 1,
\end{cases}
\end{equation}
where $t_{0}$ and $t_{1}$ are defined by \eqref{eq:t2} for $j=0$.
\end{rem}

We know the following theorem from \cite[\S2]{Lit-A} 
(see also \cite[Theorem 2.4]{NS-Tensor}); 
for the definition of crystals, 
see \cite[\S7.2]{Kas-OnC} or \cite[\S4.5]{HK} for example. 
%
%
\begin{thm} \label{thm:pc01}
\mbox{}
\begin{enu}
\item
Let $j \in I$, and $\eta \in \BP_{\cl,\,\INT}^{(j)}$. 
If $e_{j}\eta \ne \bzero$, then 
$f_{j}e_{j}\eta=\eta$. Also, 
if $f_{j}\eta \ne \bzero$, then 
$e_{j}f_{j}\eta=\eta$. 

\item
Let $\BB$ be a subset of $\BP_{\cl,\,\INT}$ 
such that the set $\BB \cup \{\bzero\}$ is stable 
under the action of the root operators $e_{j}$ and $f_{j}$ 
for all $j \in I$. The set $\BB$, equipped with 
the root operators $e_{j}$, $f_{j}$ for $j \in I$ 
and the maps \eqref{eq:wt}, \eqref{eq:ve}, 
\eqref{eq:vp}, is a crystal with weights in $P_{\cl}$.
\end{enu}
\end{thm}

\begin{rem}
In \S\ref{subsec:LS}, we wiil give a typical example of 
a subset $\BB$ of $\BP_{\cl,\,\INT}$ such that $\BB \cup \{\bzero\}$ 
is stable under the action of root operators.
\end{rem}

For each path $\eta \in \BP_{\cl}$ and $N \in \BZ_{\ge 1}$, 
we define a path $N\eta \in \BP_{\cl}$ by: 
$(N\eta)(t)=N\eta(t)$ for $t \in [0,1]$; 
by convention, we set $N\bzero:=\bzero$ for all $N \in \BZ_{\ge 1}$. 
It is easily verified that
if $\eta \in \BP_{\cl,\,\INT}^{(j)}$ for some $j \in I$, then 
$N\eta \in \BP_{\cl,\,\INT}^{(j)}$ for all $N \in \BZ_{\ge 1}$. 
%
%
\begin{lem}[%
  {see \cite[Lemma 2.4]{Lit-A} and also \cite[Lemma 2.5]{NS-Tensor}}%
 ] \label{lem:N}
Let $j \in I$. For every $\eta \in \BP_{\cl,\,\INT}^{(j)}$ and 
$N \in \BZ_{\ge 1}$, we have
\begin{align*}
& \ve_{j}(N\eta)=N\ve_{j}(\eta) \quad \text{\rm and} \quad
\vp_{j}(N\eta)=N\vp_{j}(\eta), \\
& N(e_{j}\eta)=e_{j}^{N}(N\eta) \quad \text{\rm and} \quad 
N(f_{j}\eta)=f_{j}^{N}(N\eta).
\end{align*}
\end{lem}

For $j \in I$ and $\eta \in \BP_{\cl,\,\INT}^{(j)}$, 
we define $e_{j}^{\max}\eta:=e_{j}^{\ve_{j}(\eta)}\eta 
\in \BP_{\cl,\,\INT}^{(j)}$ and 
$f_{j}^{\max}\eta:=f_{j}^{\vp_{j}(\eta)}\eta 
\in \BP_{\cl,\,\INT}^{(j)}$. The next lemma follows 
immediately from Lemma~\ref{lem:N}. 
%
%
\begin{lem} \label{lem:Ne}
Let $j \in I$. For every $\eta \in \BP_{\cl,\,\INT}^{(j)}$ and 
$N \in \BZ_{\ge 1}$, we have 
$e_{j}^{\max}(N\eta)=N(e_{j}^{\max}\eta)$ and 
$f_{j}^{\max}(N\eta)=N(f_{j}^{\max}\eta)$. 
\end{lem}

Now, for $\eta_{1},\,\eta_{2},\,\dots,\,\eta_{n} \in \BP_{\cl}$, 
define the concatenation 
$\eta_{1} \ast \eta_{2} \ast \cdots \ast 
\eta_{n} \in \BP_{\cl}$ by
%
%
\begin{align} \label{eq:cat}
& (\eta_{1} \ast \eta_{2} \ast \cdots \ast \eta_{n})(t):=
  \sum_{l=1}^{k-1} 
  \eta_{l}(1)+ \eta_{k}(nt-k+1) \nonumber \\[-1.5mm]
& \hspace{60mm} \text{for \, } \frac{k-1}{n} \le t \le 
\frac{k}{n} \text{\, and \,} 1 \le k \le n. 
\end{align}
For a subset $\BB$ of $\BP_{\cl}$ and $n \in \BZ_{\ge 1}$, 
we set 
$\BB^{\ast n}:=\bigl\{
 \eta_{1} \ast \eta_{2} \ast \cdots \ast \eta_{n} \mid 
 \text{$\eta_{k} \in \BB$ for $1 \le k \le n$}
\bigr\}$.

%
\begin{prop}[%
 {see \cite[Lemma 2.7]{Lit-A} and 
  \cite[Proposition~1.3.3]{NS-Tensor}}] \label{prop:cat}
Let $\BB$ be a subset of $\BP_{\cl,\,\INT}$ such that 
the set $\BB \cup \{\bzero\}$ is stable 
under the action of 
the root operators $e_{j}$ and $f_{j}$ for all $j \in I$; 
note that $\BB$ is a crystal with weights in $P_{\cl}$ 
by Theorem~\ref{thm:pc01}.
\begin{enu}
\item For every $n \in \BZ_{\ge 1}$, 
the set $\BB^{\ast n} \cup \{\bzero\}$ is stable 
under the root operators $e_{j}$ and $f_{j}$ for all $j \in I$. 
Therefore, $\BB^{\ast n}$ is a crystal with weights in $P_{\cl}$ 
by Theorem~\ref{thm:pc01}.

\item For every $n \in \BZ_{\ge 1}$, 
the crystal $\BB^{\ast n}$ is isomorphic as a crystal to the 
tensor product $\BB^{\otimes n}:=\BB \otimes \cdots \otimes \BB$ 
($n$ times), where the isomorphism is given by\,{\rm:}
$\eta_{1} \ast \eta_{2} \ast \cdots \ast \eta_{n} \mapsto
 \eta_{1} \otimes \eta_{2} \otimes \cdots \otimes \eta_{n}$ 
for $\eta_{1} \ast \eta_{2} \ast \cdots \ast \eta_{n} \in \BB^{\ast n}$. 
\end{enu}
\end{prop}

%
\subsection{Lakshmibai-Seshadri paths.}
\label{subsec:LS}

Let us recall the definition of 
Lakshmibai-Seshadri (LS for short) paths 
from \cite[\S4]{Lit-A}. In this subsection, 
we fix an integral weight $\lambda \in P$
which is not necessarily dominant. 
%
%
\begin{dfn} \label{dfn:Bruhat}
For $\mu,\,\nu \in W\lambda$, 
let us write $\mu \ge \nu$ if there exists a sequence 
$\mu=\mu_{0},\,\mu_{1},\,\dots,\,\mu_{n}=\nu$ 
of elements in $W\lambda$ and a sequence 
$\beta_{1},\,\dots,\,\beta_{n} \in \prr$ of 
positive real roots such that
$\mu_{k}=r_{\beta_{k}}(\mu_{k-1})$ and 
$\pair{\mu_{k-1}}{\beta^{\vee}_{k}} < 0$ 
for $k=1,\,2,\,\dots,\,n$. 
If $\mu \ge \nu$, then we define $\dist(\mu,\nu)$ to 
be the maximal length $n$ of all possible such sequences 
$\mu_{0},\,\mu_{1},\,\dots,\,\mu_{n}$ for $(\mu,\nu)$.
\end{dfn}
%
%
\begin{dfn} \label{dfn:achain}
For $\mu,\,\nu \in W\lambda$ with $\mu > \nu$ and 
a rational number $0 < \sigma < 1$, 
a $\sigma$-chain for $(\mu,\nu)$ is, 
by definition, a sequence $\mu=\mu_{0} > \mu_{1} > 
\dots > \mu_{n}=\nu$ of elements in $W\lambda$ 
such that $\dist(\mu_{k-1},\mu_{k})=1$ and 
$\sigma\pair{\mu_{k-1}}{\beta_{k}^{\vee}} \in \BZ_{< 0}$ 
for all $k=1,\,2,\,\dots,\,n$, 
where $\beta_{k}$ is the positive real root such that 
$r_{\beta_{k}}\mu_{k-1}=\mu_{k}$.
%
%
\end{dfn}
%
%
\begin{dfn} \label{dfn:LS}
An LS path of shape $\lambda \in P$ is, by definition, 
a pair $(\ud{\nu}\,;\,\ud{\sigma})$ of a sequence 
$\ud{\nu}:\nu_{1} > \nu_{2} > \cdots > \nu_{s}$ of 
elements in $W\lambda$ and a sequence 
$\ud{\sigma}:0=\sigma_{0} < \sigma_{1} < \cdots < \sigma_{s}=1$ of 
rational numbers satisfying the condition that 
there exists a $\sigma_{k}$-chain for $(\nu_{k},\,\nu_{k+1})$ 
for each $k=1,\,2,\,\dots,\,s-1$. We denote by $\BB(\lambda)$ 
the set of all LS paths of shape $\lambda$.
\end{dfn}

Let $\pi=(\nu_{1},\,\nu_{2},\,\dots,\,\nu_{s}\,;\,
\sigma_{0},\,\sigma_{1},\,\dots,\,\sigma_{s})$ be a pair of 
a sequence $\nu_{1},\,\nu_{2},\,\dots,\,\nu_{s}$ of integral weights 
with $\nu_{k} \ne \nu_{k+1}$ for $1 \le k \le s-1$ and 
a sequence $0=\sigma_{0} < \sigma_{1} < \cdots < \sigma_{s}=1$ 
of rational numbers. We identify $\pi$ with 
the following piecewise-linear, continuous map 
$\pi:[0,1] \rightarrow \BR \otimes_{\BZ} P$: 
%
%
\begin{equation} \label{eq:path}
\pi(t)=\sum_{l=1}^{k-1}
(\sigma_{l}-\sigma_{l-1})\nu_{l}+
(t-\sigma_{k-1})\nu_{k} \quad 
\text{for $\sigma_{k-1} \le t \le \sigma_{k}$, $1 \le k \le s$}. 
\end{equation}
%
%
\begin{rem} \label{rem:str01}
It is obvious from the definition that for each $\nu \in W\lambda$, 
$\pi_{\nu}:=(\nu\,;\,0,1)$ is an LS path of shape $\lambda$, 
which corresponds (under \eqref{eq:path}) to the straight line 
$\pi_{\nu}(t)=t\nu$, $t \in [0,1]$, connecting $0$ to $\nu$. 
\end{rem}

For each $\pi \in \BB(\lambda)$, we define $\cl(\pi):[0,1] 
\rightarrow \BR \otimes_{\BZ} P_{\cl}$ by: 
$(\cl(\pi))(t)=\cl(\pi(t))$ for $t \in [0,1]$. We set 
\begin{equation*} 
\BB(\lambda)_{\cl}:=
\bigl\{\cl(\pi) \mid \pi \in \BB(\lambda)\bigr\}. 
\end{equation*}
We know from \cite[\S3.1]{NS-Tensor} that 
$\BB(\lambda)_{\cl}$ is a subset of $\BP_{\cl,\,\INT}$ 
such that $\BB(\lambda)_{\cl} \cup \{\bzero\}$ is 
stable under the action of the root operators $e_{j}$ and $f_{j}$ for 
all $j \in I$. In particular, $\BB(\lambda)_{\cl}$ is 
a crystal with weights in $P_{\cl}$ by Theorem~\ref{thm:pc01}. 

Here we recall the notion of simple crystals.
A crystal $B$ with weights in $P_{\cl}$ is said to be regular 
if for every proper subset $J \subsetneq I$, 
$B$ is isomorphic as a crystal for $U_{q}(\Fg_{J})$ 
to the crystal basis of a finite-dimensional $U_{q}(\Fg_{J})$-module, 
where $\Fg_{J}$ is the (finite-dimensional) 
Levi subalgebra of $\Fg$ corresponding to $J$ (see \cite[\S2.2]{Kas-OnL}). 
A regular crystal $B$ with weights $P_{\cl}$ is 
said to be simple if the set of extremal elements in $B$ 
coincides with a $W$-orbit in $B$ through an (extremal) element in $B$ 
(cf. \cite[Definition~4.9]{Kas-OnL}). 

\begin{rem} \label{rem:simple}
\mbox{}
\begin{enu}
\item
The crystal graph of a simple crystal is connected 
(see \cite[Lemma~4.10]{Kas-OnL}). 

\item 
A tensor product of simple crystals is also a simple crystal 
(see \cite[Lemma~4.11]{Kas-OnL}).
\end{enu}
\end{rem}

We know the following theorem from 
\cite[Proposition~5.8]{NS-IMRN}, 
\cite[Theorem 2.1.1 and Proposition~3.4.2]{NS-Adv}, and 
\cite[Theorem 3.2]{NS-Tensor}. 
%
%
\begin{thm} \label{thm:LScl}
\mbox{}
\begin{enu}
\item For each $i \in I_{0}$, the crystal $\BB(\vpi_{i})_{\cl}$ 
is isomorphic, as a crystal with weights in $P_{\cl}$, 
to the crystal basis of the level-zero fundamental 
representation $W(\vpi_{i})$, 
introduced in \cite[Theorem~5.17]{Kas-OnL}, 
of the quantum affine algebra $U_{q}'(\Fg)$. 
In particular, $\BB(\vpi_{i})_{\cl}$ is a simple crystal. 

\item Let $i_{1},\,i_{2},\,\dots,\,i_{p}$ be an arbitrary sequence of 
elements of $I_{0}$ (with repetitions allowed), and set 
$\lambda:=\vpi_{i_{1}}+\vpi_{i_{2}}+ \cdots + \vpi_{i_{p}}$. 
The crystal $\BB(\lambda)_{\cl}$
is isomorphic, as a crystal with weights in $P_{\cl}$, 
to the tensor product 
$\BB(\vpi_{i_{1}})_{\cl} \otimes \BB(\vpi_{i_{2}})_{\cl} \otimes 
\cdots \otimes \BB(\vpi_{i_{p}})_{\cl}$. In particular, 
$\BB(\lambda)_{\cl}$ is also a simple crystal by Remark~\ref{rem:simple}\,(2).
\end{enu}
\end{thm}

%
\begin{rem} \label{rem:extremal}
Let $\lambda \in \sum_{i \in I_{0}}\BZ_{\ge 0}\vpi_{i}$ be 
a level-zero dominant integral weight. 
\begin{enu}
\item It can be easily seen from Remark~\ref{rem:str01} that 
$\eta_{\mu}(t):=t\mu$ is contained in $\BB(\lambda)_{\cl}$ 
for all $\mu \in \cl(W\lambda)=W_{0}\cl(\lambda)$. 

\item We know from \cite[Lemma~3.19]{NS-Tensor} that 
$\eta_{\cl(\lambda)} \in \BB(\lambda)_{\cl}$ is 
an extremal element in the sense of \cite[\S3.1]{Kas-OnL}. 
Therefore, it follows from \cite[Lemma~1.5]{AK} and 
the definition of simple crystals that 
for each $\eta \in \BB(\lambda)_{\cl}$, there exist 
$j_{1},\,j_{2},\,\dots,\,j_{p} \in I$ such that 
\begin{equation*}
e_{j_{p}}^{\max} \cdots e_{j_{2}}^{\max}e_{j_{1}}^{\max}\eta=
\eta_{\cl(\lambda)}.
\end{equation*}
Also, by the same argument as for \cite[Lemma~1.5]{AK}, 
we can show that 
for each $\eta \in \BB(\lambda)_{\cl}$, there exist 
$k_{1},\,k_{2},\,\dots,\,k_{q} \in I$ such that 
\begin{equation*}
f_{k_{q}}^{\max} \cdots f_{k_{2}}^{\max}f_{k_{1}}^{\max}\eta=
\eta_{\cl(\lambda)}.
\end{equation*}
\end{enu}
\end{rem}
%
%
\begin{lem} \label{lem:Nlambda}
Let $\lambda \in \sum_{i \in I_{0}}\BZ_{\ge 0}\vpi_{i}$ be 
a level-zero dominant integral weight, and let $n \in \BZ_{\ge 1}$. 
Then, the set $\BB(\lambda)_{\cl}^{\ast n}$ 
is identical to $\BB(n\lambda)_{\cl}$. 
\end{lem}

\begin{proof}
First, let us show the inclusion
$\BB(\lambda)_{\cl}^{\ast n} \supset \BB(n\lambda)_{\cl}$. 
It is easily seen that the element 
$\eta_{\cl(\lambda)} \ast \cdots \ast 
\eta_{\cl(\lambda)} \in \BB(\lambda)_{\cl}^{\ast n}$ is identical to 
$\eta_{\cl(n\lambda)}$. Hence it follows that the crystal 
$\BB(\lambda)_{\cl}^{\ast n}$ contains the connected component 
containing $\eta_{\cl(n\lambda)} \in \BB(n\lambda)_{\cl}$. 
Here we recall that the crystal $\BB(n\lambda)_{\cl}$ is simple 
(see Theorem~\ref{thm:LScl}), and hence connected 
(see Remark~\ref{rem:simple}\,(1)). 
Therefore, the connected component above 
is identical to $\BB(n\lambda)_{\cl}$. Thus, we have shown the inclusion
$\BB(\lambda)_{\cl}^{\ast n} \supset \BB(n\lambda)_{\cl}$. 

Now, it follows from Proposition~\ref{prop:cat} that 
$\BB(\lambda)_{\cl}^{\ast n}$ is isomorphic as a crystal to 
the tensor product $\BB(\lambda)_{\cl}^{\otimes n}$. 
Therefore, 
$\BB(\lambda)_{\cl}^{\ast n} \cong \BB(\lambda)_{\cl}^{\otimes n}$ is 
a simple crystal by Theorem~\ref{thm:LScl}\,(2) and Remark~\ref{rem:simple}\,(2), 
and hence connected by Remark~\ref{rem:simple}\,(1). 
From this, we conclude that $\BB(\lambda)_{\cl}^{\ast n}=
\BB(n\lambda)_{\cl}$, as desired.
\end{proof}

%
\subsection{Characterization of the set $\BB(\lambda)_{\cl}$.}
\label{subsec:charls}
%
%
\begin{thm} \label{thm:charls}
Let $\lambda \in \sum_{i \in I_{0}}\BZ_{\ge 0}\vpi_{i}$ be 
a level-zero dominant integral weight. 
If a subset $\BB$ of $\BP_{\cl,\,\INT}$ satisfies the following two
conditions, then the set $\BB$ is identical to $\BB(\lambda)_{\cl}$.
\begin{enumerate}
\renewcommand{\labelenumi}{\rm (\alph{enumi})}%
\item The set $\BB \cup \{\bzero\}$ is stable 
under the action of the root operators $f_{j}$ for all $j \in I$. 

\item For each $\eta \in \BB$, there exist a sequence 
$\mu_{1},\,\mu_{2},\,\dots,\,\mu_{s}$ of elements in 
$\cl(W\lambda)=W_{0}\cl(\lambda)$ and 
a sequence $0=\sigma_{0} < \sigma_{1} < \cdots < \sigma_{s}=1$ of 
rational numbers such that 
\begin{equation} \label{eq:charls}
\eta(t)=\sum_{l=1}^{k-1}
(\sigma_{l}-\sigma_{l-1})\mu_{l}+
(t-\sigma_{k-1})\mu_{k} \quad 
\text{\rm for $\sigma_{k-1} \le t \le \sigma_{k}$, $1 \le k \le s$}. 
\end{equation}
\end{enumerate}
\end{thm}

\begin{rem}
The equality $\BB=\BB(\lambda)_{\cl}$ also holds 
when we replace the root operators $f_{j}$ for $j \in I$ by $e_{j}$ 
for $j \in I$ in the theorem above; for its proof, simply replace $f_{j}$'s 
by $e_{j}$'s in the proof below. 
\end{rem}

\begin{proof}[Proof of Theorem~\ref{thm:charls}]
First, let us show the inclusion $\BB \subset \BB(\lambda)_{\cl}$. 
Fix an element $\eta \in \BB$ arbitrarily, and assume that 
$\eta$ is of the form \eqref{eq:charls}. 
Take $N \in \BZ_{\ge 1}$ such that $N\sigma_{u} \in \BZ$ 
for all $0 \le u \le s$. Then, the element 
$N\eta \in \BP_{\cl,\,\INT}$ is of the form:
\begin{equation*}
N\eta=
 \underbrace{\eta_{\mu_{1}} \ast \cdots \ast \eta_{\mu_{1}}}_{%
   \text{$N(\sigma_{1}-\sigma_{0})$-times}} \ast
 \underbrace{\eta_{\mu_{2}} \ast \cdots \ast \eta_{\mu_{2}}}_{%
   \text{$N(\sigma_{2}-\sigma_{1})$-times}} \ast \cdots \ast 
 \underbrace{\eta_{\mu_{s}} \ast \cdots \ast \eta_{\mu_{s}}}_{%
   \text{$N(\sigma_{s}-\sigma_{s-1})$-times}}.
\end{equation*}
Since $\eta_{\mu} \in \BB(\lambda)_{\cl}$ 
for every $\mu \in \cl(W\lambda)$ (see Remark~\ref{rem:extremal}\,(1)), 
we have $N\eta \in \BB(\lambda)_{\cl}^{\ast N}$, and hence
$N\eta \in \BB(N\lambda)_{\cl}$ by Lemma~\ref{lem:Nlambda}. 
By Remark~\ref{rem:extremal}, 
there exists $k_{1},\,k_{2},\,\dots,\,k_{q} \in I$ such that 
\begin{equation*}
f_{k_{q}}^{\max} \cdots 
f_{k_{2}}^{\max}f_{k_{1}}^{\max}(N\eta)=
\eta_{\cl(N\lambda)}.
\end{equation*}
Using Lemma~\ref{lem:Ne} and condition (a) repeatedly, 
we deduce that 
\begin{equation*}
f_{k_{q}}^{\max} \cdots 
f_{k_{2}}^{\max}f_{k_{1}}^{\max}(N\eta) =
N(f_{k_{q}}^{\max} \cdots 
f_{k_{2}}^{\max}f_{k_{1}}^{\max}\eta).
\end{equation*}
Combining these equalities, we obtain
$N(f_{k_{q}}^{\max} \cdots 
f_{k_{2}}^{\max}f_{k_{1}}^{\max}\eta) = \eta_{\cl(N\lambda)}$. 
Since $\eta_{\cl(N\lambda)}=N\eta_{\cl(\lambda)}$, we get 
\begin{equation} \label{eq:charls01}
f_{k_{q}}^{\max} \cdots 
f_{k_{2}}^{\max}f_{k_{1}}^{\max}
\eta=\eta_{\cl(\lambda)} \in \BB(\lambda)_{\cl}.
\end{equation}
Therefore, by Theorem~\ref{thm:pc01}\,(1), 
$\eta=
 e_{k_{1}}^{c_{1}}e_{k_{2}}^{c_{2}} \cdots 
 e_{k_{q}}^{c_{q}}\eta_{\cl(\lambda)} \in \BB(\lambda)_{\cl}$ 
for some $c_{1},\,c_{2},\,\dots,\,c_{q} \in \BZ_{\ge 0}$. 
Thus we have shown the inclusion $\BB \subset \BB(\lambda)_{\cl}$. Also, 
we should remark that $\eta_{\cl(\lambda)} \in \BB$ 
by \eqref{eq:charls01} and condition (a). 

Next, let us show the opposite inclusion $\BB \supset \BB(\lambda)_{\cl}$. 
Fix an element $\eta' \in \BB(\lambda)_{\cl}$ arbitrarily. 
By Remark~\ref{rem:extremal}, 
there exists $j_{1},\,j_{2},\,\dots,\,j_{p} \in I$ such that 
\begin{equation*}
e_{j_{p}}^{\max} \cdots 
e_{j_{2}}^{\max}e_{j_{1}}^{\max}\eta'=
\eta_{\cl(\lambda)}.
\end{equation*}
Therefore, by Theorem~\ref{thm:pc01}\,(1), 
$\eta'=
 f_{j_{1}}^{d_{1}}f_{j_{2}}^{d_{2}} \cdots 
 f_{j_{p}}^{d_{p}}\eta_{\cl(\lambda)}$
for some $d_{1},\,d_{2},\,\dots,\,d_{p} \in \BZ_{\ge 0}$. 
Since $\eta_{\cl(\lambda)} \in \BB$ as shown above, 
it follows from condition (a) that $\eta' \in \BB$. 
Thus we have shown the inclusion $\BB \supset \BB(\lambda)_{\cl}$, 
thereby completing the proof of the theorem. 
\end{proof}

%
\section{Quantum Lakshmibai-Seshadri paths.}
\label{sec:QLS}

%
\subsection{Quantum Bruhat graph.}
\label{subsec:def-QBG}

In this subsection, we fix a subset $J$ of $I_{0}$. Set
\begin{equation*}
W_{J}:=\langle r_{j} \mid j \in J \rangle \subset W_{0}.
\end{equation*}
It is well-known that each coset in $W_{0}/W_{J}$ has a unique element of 
minimal length, called the minimal coset representative for the coset; 
we denote by $W_{0}^{J} \subset W_{0}$ 
the set of minimal coset representatives for the cosets in 
$W_{0}/W_{J}$, and by $\mcr{\,\cdot\,}=\mcr{\,\cdot\,}_{J}:
W_{0} \twoheadrightarrow W_{0}^{J} \cong W_{0}/W_{J}$
the canonical projection. 
Also, we set $\Delta_{J}:=\Delta_{0} \cap 
\bigl(\bigoplus_{j \in J} \BZ \alpha_{j}\bigr)$, 
$\Delta_{J}^{\pm}:=
\Delta_{0}^{\pm} \cap \bigl(\bigoplus_{j \in J} \BZ \alpha_{j}\bigr)$, 
and 
$\rho:=(1/2)\sum_{\alpha \in \Delta_{0}^{+}}\alpha$, 
$\rho_{J}:=(1/2)\sum_{\alpha \in \Delta_{J}^{+}} \alpha$. 
%
%
\begin{dfn} \label{dfn:QBG}
The (parabolic) quantum Bruhat graph is 
a $(\Delta_{0}^{+} \setminus \Delta_{J}^{+})$-labeled, 
directed graph with vertex set $W_{0}^{J}$ 
and $(\Delta_{0}^{+} \setminus \Delta_{J}^{+})$-labeled, directed edges of 
the following form: $\mcr{wr_{\beta}} \stackrel{\beta}{\leftarrow} w$ 
for $w \in W_{0}^{J}$ and 
$\beta \in \Delta_{0}^{+} \setminus \Delta_{J}^{+}$ such that either
\begin{enumerate}
\renewcommand{\labelenumi}{(\roman{enumi})}

\item $\ell(\mcr{wr_{\beta}})=\ell(w)+1$, or 

\item $\ell(\mcr{wr_{\beta}})=\ell(w)-2\pair{\rho-\rho_{J}}{\beta^{\vee}}+1$;
\end{enumerate}
if (i) holds (resp., (ii) holds), then the edge is called a Bruhat edge 
(resp., a quantum edge). 
\end{dfn}

%
\begin{rem} \label{rem:QBG1}
If $w \in W_{0}^{J}$ and 
$\beta \in \Delta_{0}^{+} \setminus \Delta_{J}^{+}$ satisfy 
the condition that $\ell(\mcr{wr_{\beta}})=\ell(w)+1$, 
then $wr_{\beta} \in W_{0}^{J}$. Indeed, 
since $\ell(wr_{\beta}) \ge \ell(\mcr{wr_{\beta}})=\ell(w)+1$, 
it follows that $wr_{\beta}$ is greater than $w$ in the ordinary Bruhat order. 
Therefore, by \cite[Proposition 2.5.1]{BB}, $\mcr{wr_{\beta}}$ 
is greater than or equal to $\mcr{w}=w$ in the ordinary Bruhat order. 
Since $\ell(\mcr{wr_{\beta}})=\ell(w)+1$ by the assumption, 
there exists $\gamma \in \Delta_{0}^{+}$ such that 
$\mcr{wr_{\beta}}=wr_{\gamma}$. 
Now, we take a dominant integral weight $\Lambda \in P_{\cl}$ 
with respect to the finite root system $\Delta_{0}$ such that 
$\bigl\{j \in I_{0} \mid \pair{\Lambda}{\alpha_{j}^{\vee}}=0\bigr\}=J$; 
note that $\pair{\Lambda}{\beta^{\vee}} > 0$ since 
$\beta \in \Delta_{0}^{+} \setminus \Delta_{J}^{+}$. 
Then we have
$wr_{\beta}\Lambda=\mcr{wr_{\beta}}\Lambda=wr_{\gamma}\Lambda$, 
and hence $r_{\beta}\Lambda=r_{\gamma}\Lambda$. 
It follows that $\pair{\Lambda}{\beta^{\vee}}\beta=
\pair{\Lambda}{\gamma^{\vee}}\gamma$. 
Since $\beta$ and $\gamma$ are both contained in $\Delta_{0}^{+}$, and 
since $\pair{\Lambda}{\beta^{\vee}} > 0$, 
we deduce that $\beta=\gamma$. 
Thus, we obtain $\mcr{wr_{\beta}}=wr_{\gamma}=wr_{\beta}$, 
which implies that $wr_{\beta} \in W_{0}^{J}$. 
%
\end{rem}

%
%
\begin{rem} \label{rem:QBG}
We know from \cite[Lemma 10.18]{LS} that 
the condition (ii) above is equivalent to the following condition : 
\begin{enumerate}
\renewcommand{\labelenumi}{(\roman{enumi})}
\setcounter{enumi}{2}

\item $\ell(\mcr{wr_{\beta}})=
\ell(w)-2\pair{\rho-\rho_{J}}{\beta^{\vee}}+1$ and 
$\ell(wr_{\beta})=\ell(w)-2\pair{\rho}{\beta^{\vee}}+1$. 

\end{enumerate}
%
%
%
%
%
%
\end{rem}

Let $x,\,y \in W_{0}^{J}$. A directed path $\bd$
from $y$ to $x$ in the parabolic quantum Bruhat graph is, 
by definition, a pair of a sequence $w_{0},\,w_{1},\,\dots,\,w_{n}$ 
of elements in $W_{0}^{J}$ and a sequence 
$\beta_{1},\,\beta_{2},\,\dots,\,\beta_{n}$ of 
elements in $\Delta_{0}^{+} \setminus \Delta_{J}^{+}$ 
such that
%
%
\begin{equation} \label{eq:dp}
\bd : 
x=
 w_{0} \stackrel{\beta_{1}}{\leftarrow}
 w_{1} \stackrel{\beta_{2}}{\leftarrow} \cdots 
       \stackrel{\beta_{n}}{\leftarrow}
 w_{n}=y. 
\end{equation}
A directed path $\bd$ from $y$ to $x$ said to be 
shortest if its length $n$ is minimal among 
all possible directed paths from $y$ to $x$. Denote by 
$\len{x}{y}$ the length of a shortest directed path 
from $y$ to $x$ in the parabolic quantum Bruhat graph.
%
%
\subsection{Definition of quantum Lakshmibai-Seshadri paths.}
\label{subsec:QLS}

In this subsection, 
we fix a level-zero dominant integral weight
$\lambda \in \sum_{i \in I_{0}} \BZ_{\ge 0} \vpi_{i}$, and 
set $\Lambda:=\cl(\lambda)$ for simplicity of notation. 
Also, we set 
\begin{equation*}
J:=\bigl\{j \in I_{0} \mid \pair{\Lambda}{\alpha_{j}^{\vee}}=0 \bigr\} 
 \subset I_{0}.
\end{equation*}
%
%
\begin{dfn} \label{dfn:QBG-achain}
Let $x,\,y \in W_{0}^{J}$, and let $\sigma \in \BQ$ 
be such that $0 < \sigma < 1$. A directed $\sigma$-path 
from $y$ to $x$ is, by definition, a directed path 
\begin{equation*}
x=w_{0} \stackrel{\beta_{1}}{\leftarrow} w_{1}
\stackrel{\beta_{2}}{\leftarrow} w_{2} 
\stackrel{\beta_{3}}{\leftarrow} \cdots 
\stackrel{\beta_{n}}{\leftarrow} w_{n}=y
\end{equation*}
from $y$ to $x$ in the parabolic quantum Bruhat graph 
satisfying the condition that 
\begin{equation*}
\sigma \pair{\Lambda}{\beta_{k}^{\vee}} \in \BZ 
\quad \text{for all $1 \le k \le n$}.
\end{equation*}
\end{dfn}
%
%
\begin{dfn} \label{dfn:qLS}
Denote by $\ti{\BB}(\lambda)_{\cl}$ (resp., $\ha{\BB}(\lambda)_{\cl}$) 
the set of all pairs $\eta=(\ud{x}\,;\,\ud{\sigma})$ of 
a sequence $\ud{x}\,:\,x_{1},\,x_{2},\,\dots,\,x_{s}$ of 
elements in $W_{0}^{J}$, with $x_{k} \ne x_{k+1}$ 
for $1 \le k \le s-1$, and a sequence 
$\ud{\sigma}\,:\,
 0=\sigma_{0} < \sigma_{1} < \cdots < \sigma_{s}=1$ of rational numbers
 satisfying the condition that 
there exists a directed $\sigma_{k}$-path 
(resp., a directed $\sigma_{k}$-path of length $\len{x_{k}}{x_{k+1}}$) 
from $x_{k+1}$ to $x_{k}$ for each $1 \le k \le s-1$; 
observe that $\ha{\BB}(\lambda)_{\cl} \subset 
\ti{\BB}(\lambda)_{\cl}$. We call an element of 
$\ti{\BB}(\lambda)_{\cl}$ a quantum Lakshmibai-Seshadri path 
of shape $\lambda$. 
\end{dfn}

Let $\eta=(x_{1},\,x_{2},\,\dots,\,x_{s}\,;\,
\sigma_{0},\,\sigma_{1},\,\dots,\,\sigma_{s})$ 
be a rational path, that is, 
a pair of a sequence $x_{1},\,x_{2},\,\dots,\,x_{s}$ 
of elements in $W_{0}^{J}$, with $x_{k} \ne x_{k+1}$ 
for $1 \le k \le s-1$, and a sequence 
$0=\sigma_{0} < \sigma_{1} < \cdots < \sigma_{s}=1$ of 
rational numbers. 
We identify $\eta$ with the following piecewise-linear, 
continuous map 
$\eta:[0,1] \rightarrow \BR \otimes_{\BZ} P_{\cl}$ 
(cf. \eqref{eq:path}): 
%
%
\begin{equation} \label{eq:QBG-path}
\eta(t)=\sum_{l=1}^{k-1}
(\sigma_{l}-\sigma_{l-1})x_{l}\Lambda+
(t-\sigma_{k-1})x_{k}\Lambda \quad 
\text{for $\sigma_{k-1} \le t \le \sigma_{k}$, $1 \le k \le s$};
\end{equation}
note that the map $W_{0}^{J} \rightarrow W_{0}\Lambda$, 
$w \mapsto w\Lambda$, is bijective. 
We will prove that under this identification, 
both $\ti{\BB}(\lambda)_{\cl}$ and $\ha{\BB}(\lambda)_{\cl}$ 
can be regarded as a subset of $\BP_{\cl,\,\INT}$ 
(see Proposition~\ref{prop:ip}). Furthermore, we will prove that 
both of the sets $\ti{\BB}(\lambda)_{\cl} \cup \{\bzero\}$ and 
$\ha{\BB}(\lambda)_{\cl} \cup \{\bzero\}$ are stable under the action of 
root operators (see Proposition~\ref{prop:stable}).
%
%
\section{Main result.}
\label{sec:main}
%
%
\subsection{Statement and some technical lemmas.}
\label{subsec:prf-main1}
Throughout this section, 
we fix a level-zero dominant integral weight
$\lambda \in \sum_{i \in I_{0}} \BZ_{\ge 0} \vpi_{i}$.
Set $\Lambda:=\cl(\lambda)$, and 
\begin{equation*}
J:=\bigl\{j \in I_{0} \mid \pair{\Lambda}{\alpha_{j}^{\vee}}=0 \bigr\} 
 \subset I_{0}.
\end{equation*}
The following theorem is the main result of this paper;
it is obtained as a by-product of an explicit description, 
given in \S\ref{subsec:ro}, of the image of a quantum LS path 
as a rational path under the action of root operators 
on quantum LS paths. 
%
%
\begin{thm} \label{thm:main}
With the notation and setting above, we have
\begin{equation*}
\ti{\BB}(\lambda)_{\cl}=
\ha{\BB}(\lambda)_{\cl}=\BB(\lambda)_{\cl}.
\end{equation*}
\end{thm}

In view of Theorem~\ref{thm:charls}, 
in order to prove Theorem~\ref{thm:main}, 
it suffices to prove that both $\ti{\BB}(\lambda)_{\cl}$ 
and $\ha{\BB}(\lambda)_{\cl}$ are contained in $\BP_{\cl,\,\INT}$ 
(see Proposition~\ref{prop:ip} below), 
and that both of the sets 
$\ti{\BB}(\lambda)_{\cl} \cup \{\bzero\}$ and 
$\ha{\BB}(\lambda)_{\cl} \cup \{\bzero\}$ are stable 
under the action of the root operators $f_{j}$ for all $j \in I$ 
(see Proposition~\ref{prop:stable} below). 
To prove these, we need some lemmas. 

%
\begin{lem}[{\cite[Proposition 5.11]{LNSSS1}}] \label{lem:theta}
Let $w \in W_{0}^{J}$. 
If $w^{-1}\theta \in \Delta_{0}^{-}$, then there exists 
a quantum edge $\mcr{r_{\theta}w} \stackrel{-w^{-1}\theta}{\longleftarrow} 
w$ from $w$ to $\mcr{r_{\theta}w}$ in the parabolic quantum Bruhat graph. 
\end{lem}
%
%
\begin{lem}[{\cite[Proposition 5.10\,(1) and (3)]{LNSSS1}}] \label{lem:mcr}
Let $w \in W_{0}^{J}$ and $j \in I_{0}$. 
If $w^{-1}\alpha_{j} \in \Delta_{0} \setminus \Delta_{J}$, 
then $r_{j}w \in W_{0}^{J}$. 
\end{lem}

%
\begin{lem} \label{lem:dist1}
Let $w \in W_{0}^{J}$ and $\beta \in \Delta_{0}^{+} \setminus \Delta_{J}^{+}$ 
be such that $\mcr{wr_{\beta}} \stackrel{\beta}{\leftarrow} w$. 
Let $j \in I_{0}$. 
\begin{enu}
\item If $\pair{w\Lambda}{\alpha_{j}^{\vee}} > 0$ and 
$w\beta \ne \pm \alpha_{j}$, then 
$\pair{wr_{\beta}\Lambda}{\alpha_{j}^{\vee}} > 0$. Also, 
both $r_{j}\mcr{wr_{\beta}}$ and $r_{j}w$ are 
contained in $W_{0}^{J}$, and 
$r_{j}\mcr{wr_{\beta}} \stackrel{\beta}{\leftarrow} r_{j}w$. 

\item If $\pair{wr_{\beta}\Lambda}{\alpha_{j}^{\vee}} < 0$ and 
$w\beta \ne \pm \alpha_{j}$, then 
$\pair{w\Lambda}{\alpha_{j}^{\vee}} < 0$. Also, 
both $r_{j}\mcr{wr_{\beta}}$ and $r_{j}w$ are 
contained in $W_{0}^{J}$, and 
$r_{j}\mcr{wr_{\beta}} \stackrel{\beta}{\leftarrow} r_{j}w$. 

\item If $\pair{wr_{\beta}\Lambda}{\alpha_{j}^{\vee}} < 0$ and 
$\pair{w\Lambda}{\alpha_{j}^{\vee}} \ge 0$, 
then $w\beta = \pm \alpha_{j}$. 

\item If $\pair{wr_{\beta}\Lambda}{\alpha_{j}^{\vee}} \le 0$ and 
$\pair{w\Lambda}{\alpha_{j}^{\vee}} > 0$, 
then $w\beta = \pm \alpha_{j}$. 
\end{enu}
\end{lem}

\begin{proof}
(1) Since $\pair{w\Lambda}{\alpha_{j}^{\vee}} > 0$, 
we see that $w^{-1}\alpha_{j} \in \Delta_{0}^{+} \setminus \Delta_{J}^{+}$. 
By \cite[Proposition~5.10\,(3)]{LNSSS1}, 
there exists a Bruhat edge 
$r_{j}w \stackrel{w^{-1}\alpha_{j}}{\longleftarrow} w$ 
in the parabolic quantum Bruhat graph, with $r_{j}w \in W_{0}^{J}$. 
If the edge $\mcr{wr_{\beta}} \stackrel{\beta}{\leftarrow} w$ is 
a Bruhat (resp., quantum) edge, then it follows from 
the left diagram of (5.3) (resp., (5.4)) in part (1) (resp., part (2)) of 
\cite[Lemma 5.14]{LNSSS1} that 
$r_{j}\mcr{wr_{\beta}}=\mcr{r_{j}wr_{\beta}} \in W_{0}^{J}$, and 
there exists a Bruhat (resp., quantum) edge $r_{j}\mcr{wr_{\beta}}
\stackrel{\beta}{\longleftarrow} r_{j}w$ and a Bruhat edge 
$r_{j}\mcr{wr_{\beta}} 
 \stackrel{\mcr{wr_{\beta}}^{-1}\alpha_{j}}{\longleftarrow} 
 \mcr{wr_{\beta}}$ 
in the parabolic quantum Bruhat graph. 
In particular, we have $\mcr{wr_{\beta}}^{-1}\alpha_{j} \in 
\Delta_{0}^{+} \setminus \Delta_{J}^{+}$, which implies that 
$\pair{wr_{\beta}\Lambda}{\alpha_{j}^{\vee}} > 0$. 
This proves part (1). 

(2) Since $\pair{wr_{\beta}\Lambda}{\alpha_{j}^{\vee}} < 0$, 
we see that $\mcr{wr_{\beta}}^{-1}\alpha_{j} \in 
\Delta_{0}^{-} \setminus \Delta_{J}^{-}$. 
By \cite[Proposition~5.10\,(1)]{LNSSS1}, 
there exists a Bruhat edge 
$\mcr{wr_{\beta}} \stackrel{-\mcr{wr_{\beta}}^{-1}\alpha_{j}}{\longleftarrow} 
r_{j}\mcr{wr_{\beta}}$ in the parabolic quantum Bruhat graph, 
with $r_{j}\mcr{wr_{\beta}} \in W_{0}^{J}$. 
If the edge $\mcr{wr_{\beta}} \stackrel{\beta}{\leftarrow} w$ is 
a Bruhat (resp., quantum) edge, then it follows from 
the right diagram of (5.3) (resp., (5.4)) in part (1) 
(resp., part (2)) of \cite[Lemma 5.14]{LNSSS1} 
that $r_{j}w \in W_{0}^{J}$, and there exists 
a Bruhat (resp., quantum) edge $r_{j}\mcr{wr_{\beta}} 
\stackrel{\beta}{\longleftarrow} r_{j}w$ and a Bruhat edge 
$w \stackrel{-w^{-1}\alpha_{j}}{\longleftarrow} r_{j}w$ 
in the parabolic quantum Bruhat graph. 
In particular, we have $w^{-1}\alpha_{j} \in 
\Delta_{0}^{-} \setminus \Delta_{J}^{-}$, which implies that 
$\pair{w\Lambda}{\alpha_{j}^{\vee}} < 0$. 
This proves part~(2). 

(3) (resp., (4)) 
Assume that $\pair{wr_{\beta}\Lambda}{\alpha_{j}^{\vee}} < 0$ and 
$\pair{w\Lambda}{\alpha_{j}^{\vee}} \ge 0$ 
(resp., $\pair{wr_{\beta}\Lambda}{\alpha_{j}^{\vee}} \le 0$ and 
$\pair{w\Lambda}{\alpha_{j}^{\vee}} > 0$). 
Suppose that $w\beta \ne \pm \alpha_{j}$.
Then it follows from part (2) (resp., (1)) 
that $\pair{w\Lambda}{\alpha_{j}^{\vee}} < 0$ 
(resp., $\pair{wr_{\beta}\Lambda}{\alpha_{j}^{\vee}} > 0$), 
which is a contradiction. Thus we get $w\beta = \pm \alpha_{j}$. 
This completes the proof of Lemma~\ref{lem:dist1}. 
\end{proof}

%
%
%

%
\begin{lem} \label{lem:dist2}
Let $w \in W_{0}^{J}$ and $\beta \in \Delta_{0}^{+} \setminus \Delta_{J}^{+}$ 
be such that $\mcr{wr_{\beta}} \stackrel{\beta}{\leftarrow} w$. 
Let $z \in W_{J}$ be such that $r_{\theta}w=\mcr{r_{\theta}w}z$; 
note that $z\beta \in \Delta_{0}^{+} \setminus \Delta_{J}^{+}$. 
\begin{enu}
\item If $\pair{w\Lambda}{\alpha_{0}^{\vee}} > 0$ and 
$w\beta \ne \pm \theta$, then 
$\pair{wr_{\beta}\Lambda}{\alpha_{0}^{\vee}} > 0$ and 
$\mcr{r_{\theta}wr_{\beta}}
 \stackrel{z\beta}{\leftarrow} \mcr{r_{\theta}w}$. 

\item If $\pair{wr_{\beta}\Lambda}{\alpha_{0}^{\vee}} < 0$ and 
$w\beta \ne \pm \theta$, then 
$\pair{w\Lambda}{\alpha_{0}^{\vee}} < 0$ and $\mcr{r_{\theta}wr_{\beta}}
 \stackrel{z\beta}{\leftarrow} \mcr{r_{\theta}w}$. 

\item If $\pair{wr_{\beta}\Lambda}{\alpha_{0}^{\vee}} < 0$ and 
$\pair{w\Lambda}{\alpha_{0}^{\vee}} \ge 0$, 
then $w\beta = \pm \theta$. 

\item If $\pair{wr_{\beta}\Lambda}{\alpha_{0}^{\vee}} \le 0$ and 
$\pair{w\Lambda}{\alpha_{0}^{\vee}} > 0$, 
then $w\beta = \pm \theta$. 
\end{enu}
\end{lem}

\begin{proof}
(1) Since $\pair{w\Lambda}{\alpha_{0}^{\vee}} > 0$, 
we see that $w^{-1}\theta \in \Delta_{0}^{-} \setminus \Delta_{J}^{-}$.
By \cite[Proposition~5.11\,(1)]{LNSSS1}, 
there exists a quantum edge 
$\mcr{r_{\theta}w} \stackrel{-w^{-1}\theta}{\longleftarrow} w$ 
in the parabolic quantum Bruhat graph. 
If the edge $\mcr{wr_{\beta}} \stackrel{\beta}{\leftarrow} w$ is 
a Bruhat (resp., quantum) edge, then it follows from 
the left diagram of (5.5) or (5.6) (resp., (5.7) or (5.8)) 
in part (3) (resp., part (4)) of \cite[Lemma 5.14]{LNSSS1} that 
there exists an edge $\mcr{r_{\theta}wr_{\beta}}
 \stackrel{z\beta}{\leftarrow} \mcr{r_{\theta}w}$ 
and a quantum edge 
$\mcr{r_{\theta}wr_{\beta}} 
 \stackrel{-\mcr{wr_{\beta}}^{-1}\theta}{\longleftarrow} 
 \mcr{wr_{\beta}}$ 
in the parabolic quantum Bruhat graph. 
In particular, we have
$\mcr{wr_{\beta}}^{-1}\theta \in \Delta_{0}^{-} \setminus \Delta_{J}^{-}$, 
which implies that $\pair{wr_{\beta}\Lambda}{\alpha_{0}^{\vee}} > 0$. 
This proves part (1). 

(2) Since $\pair{wr_{\beta}\Lambda}{\alpha_{0}^{\vee}} < 0$, 
we see that $\mcr{wr_{\beta}}^{-1}\theta \in 
\Delta_{0}^{+} \setminus \Delta_{J}^{+}$.
By \cite[Proposition~5.11\,(3)]{LNSSS1}, 
there exists a quantum edge 
$\mcr{wr_{\beta}} \stackrel{z'\mcr{wr_{\beta}}^{-1}\theta}{\longleftarrow} 
 \mcr{r_{\theta}wr_{\beta}}$ 
in the parabolic quantum Bruhat graph, 
where $z' \in W_{J}$ is defined by: 
$r_{\theta}\mcr{wr_{\beta}}=\mcr{r_{\theta}wr_{\beta}}z'$. 
If the edge $\mcr{wr_{\beta}} \stackrel{\beta}{\leftarrow} w$ is 
a Bruhat (resp., quantum) edge, then it follows from 
the right diagram of (5.5) or (5.6) (resp., (5.7) or (5.8)) 
in part (3) (resp., part (4)) of \cite[Lemma 5.14]{LNSSS1} 
that there exists an edge 
$\mcr{r_{\theta}wr_{\beta}}
 \stackrel{z\beta}{\leftarrow} \mcr{r_{\theta}w}$ 
and a quantum edge 
$w \stackrel{zw^{-1}\theta}{\longleftarrow} \mcr{r_{\theta}w}$ 
in the parabolic quantum Bruhat graph. 
In particular, we have 
$w^{-1}\theta \in \Delta_{0}^{+} \setminus \Delta_{J}^{+}$, 
which implies that $\pair{w\Lambda}{\alpha_{0}^{\vee}} < 0$. 
This proves part (2). 

Parts (3) and (4) can be shown by using parts (1) and (2) 
in the same way as parts (3) and (4) of Lemma~\ref{lem:dist1}. 
This completes the proof of Lemma~\ref{lem:dist2}.
\end{proof}
%
%
\begin{lem} \label{lem:sigma0}
Let $\lambda$, $\Lambda$, and $J$ be as above. 
Let $x,\,y \in W_{0}^{J}$, and let $\sigma \in \BQ$ be such that $0 < \sigma < 1$. 
Assume that there exists a directed $\sigma$-path from $y$ to $x$ as follows: 
\begin{equation*}
x=w_{0} \stackrel{\beta_{1}}{\leftarrow} w_{1}
\stackrel{\beta_{2}}{\leftarrow} w_{2} 
\stackrel{\beta_{3}}{\leftarrow} \cdots 
\stackrel{\beta_{n}}{\leftarrow} w_{n}=y.
\end{equation*}
Then, $\sigma(x\Lambda-y\Lambda)$ is contained in
$Q_{0}:=\bigoplus_{j \in I_{0}}\BZ \alpha_{j}$. 
\end{lem}

\begin{proof}
We have
\begin{align*}
\sigma (x\Lambda-y\Lambda)
 & = \sum_{k=1}^{n} \sigma (w_{k-1}\Lambda-w_{k}\Lambda)
   = \sum_{k=1}^{n} \sigma (w_{k}r_{\beta_{k}}\Lambda-w_{k}\Lambda) \\[1.5mm]
 & = -\sum_{k=1}^{n} \sigma \pair{\Lambda}{\beta_{k}^{\vee}}\,w_{k}\beta_{k}.
\end{align*}
It follows from the definition of a directed $\sigma$-path that 
$\sigma \pair{\Lambda}{\beta_{k}^{\vee}} \in \BZ$ for all $1 \le k \le n$. 
Also, it is obvious that $w_{k}\beta_{k} \in Q_{0}$ for all 
$1 \le k \le n$. Therefore, we conclude that 
$\sigma (x\Lambda-y\Lambda) \in Q_{0}$. 
This proves the lemma. 
\end{proof}
%
%
\begin{lem} \label{lem:weight}
Let $\lambda$, $\Lambda$, and $J$ be as above. 
If $\eta \in \ti{\BB}(\lambda)_{\cl}$, then 
$\eta(1)$ is contained in $\Lambda+Q_{0}$, 
and hence in $P_{\cl}$. 
\end{lem}

\begin{proof}
Let $\eta=(x_{1},\,x_{2},\,\dots,\,x_{s}\,;\,
\sigma_{0},\,\sigma_{1},\,\dots,\,\sigma_{s}) 
\in \ti{\BB}(\lambda)_{\cl}$. Then we have 
(see \eqref{eq:QBG-path})
\begin{equation*}
\eta(1)=x_{s}\Lambda+
 \sum_{k=1}^{s-1}\sigma_{k}(x_{k}\Lambda-x_{k+1}\Lambda). 
\end{equation*}
It is obvious that $x_{s}\Lambda \in \Lambda+Q_{0}$. 
Also, it follows from Lemma~\ref{lem:sigma0} that 
$\sigma_{k}(x_{k}\Lambda-x_{k+1}\Lambda) \in Q_{0}$ 
for each $1 \le k \le s-1$. Therefore, we conclude that 
$\eta(1) \in \Lambda+Q_{0}$. This proves the lemma.
\end{proof}

In what follows, we set $s_{j}:=r_{j}$ for $j \in I_{0}$, and 
$s_{0}:=r_{\theta} \in W_{0}$, in order to state our 
results and write their proofs in a way independent of whether 
$j=0$ or not. 
%
%
\begin{lem} \label{lem:sigma1}
Let $\lambda$, $\Lambda$, and $J$ be as above. 
Let $x,\,y \in W_{0}^{J}$, and assume that 
there exists a directed path
%
%
\begin{equation} \label{eq:sigma1-0}
x=w_{0} \stackrel{\beta_{1}}{\leftarrow} w_{1}
\stackrel{\beta_{2}}{\leftarrow} w_{2} 
\stackrel{\beta_{3}}{\leftarrow} \cdots 
\stackrel{\beta_{n}}{\leftarrow} w_{n}=y.
\end{equation}
from $y$ to $x$. Let $j \in I$. 
\begin{enu} 
\item If there exists $1 \le p \le n$ such that 
$\pair{w_{k}\Lambda}{\alpha_{j}^{\vee}} < 0$ for all $0 \le k \le p-1$ and 
$\pair{w_{p}\Lambda}{\alpha_{j}^{\vee}} \ge 0$, then 
$\mcr{s_{j}w_{p-1}}=w_{p}$, and 
there exists a directed path from $y$ to $\mcr{s_{j}x}$ of the form: 
%
%
\begin{equation} \label{eq:sigma1-1}
\mcr{s_{j}x}=
\mcr{s_{j}w_{0}} \stackrel{z_{1}\beta_{1}}{\leftarrow} \cdots 
\stackrel{z_{p-1}\beta_{p-1}}{\leftarrow}
\mcr{s_{j}w_{p-1}}=w_{p} \stackrel{\beta_{p+1}}{\leftarrow} \cdots 
\stackrel{\beta_{n}}{\leftarrow} w_{n}=y. 
\end{equation}
Here, if $j \in I_{0}$, then we define $z_{k} \in W_{J}$ to 
be the identity element for all $1 \le k \le p-1$; 
if $j=0$, then we define $z_{k} \in W_{J}$ by 
$r_{\theta}w_{k}=\mcr{r_{\theta}w_{k}}z_{k}$ 
for each $1 \le k \le p-1$. 

\item If the directed path \eqref{eq:sigma1-0} from $y$ to $x$ is 
shortest, i.e., $\len{x}{y}=n$, then the directed path 
\eqref{eq:sigma1-1} from $y$ to $\mcr{s_{j}x}$ is also shortest, 
i.e., $\len{\mcr{s_{j}x}}{y}=n-1$. 

\item If the directed path \eqref{eq:sigma1-0} is a 
directed $\sigma$-path from $y$ to $x$ 
for some rational number $0 < \sigma < 1$, 
then the directed path \eqref{eq:sigma1-1} is a 
directed $\sigma$-path from $y$ to $\mcr{s_{j}x}$.
\end{enu}
\end{lem}

\begin{proof}
(1) We give a proof only for the case $j \in I_{0}$. 
The proof for the case $j=0$ is similar; 
replace $\alpha_{j}$ and $\alpha_{j}^{\vee}$ by $-\theta$ and $-\theta^{\vee}$, 
respectively, and use Lemma~\ref{lem:dist2} instead of Lemma~\ref{lem:dist1}.
First, let us check that $w_{k}\beta_{k} \ne \pm \alpha_{j}$ 
for any $1 \le k \le p-1$. Suppose, contrary to our claim, that 
$w_{k}\beta_{k} = \pm \alpha_{j}$ for some $1 \le k \le p-1$. 
Then, 
\begin{equation*}
w_{k-1}\Lambda=w_{k}r_{\beta_{k}}\Lambda=
r_{w_{k}\beta_{k}}w_{k}\Lambda=s_{j}w_{k}\Lambda, 
\end{equation*}
and hence $\pair{w_{k-1}\Lambda}{\alpha_{j}^{\vee}}=
\pair{s_{j}w_{k}\Lambda}{\alpha_{j}^{\vee}}=
-\pair{w_{k}\Lambda}{\alpha_{j}^{\vee}} > 0$, which contradicts our assumption. 
Thus,  $w_{k}\beta_{k} \ne \pm \alpha_{j}$ for any $1 \le k \le p-1$.
It follows from Lemma~\ref{lem:dist1}\,(2) and our assumption that 
$\mcr{s_{j}w_{k-1}} \stackrel{\beta_{k}}{\leftarrow} \mcr{s_{j}w_{k}}$ 
for all $1 \le k \le p-1$. Also, since $\pair{w_{p-1}\Lambda}{\alpha_{j}^{\vee}} < 0$
and $\pair{w_{p}\Lambda}{\alpha_{j}^{\vee}} \ge 0$, it follows from 
Lemma~\ref{lem:dist1}\,(3) that $w_{p}\beta_{p}=\pm \alpha_{j}$, 
and hence 
\begin{equation*}
s_{j}w_{p-1}\Lambda=s_{j}w_{p}r_{\beta_{p}}\Lambda=
s_{j}r_{w_{p}\beta_{p}}w_{p}\Lambda=
s_{j}s_{j}w_{p}\Lambda=w_{p}\Lambda.
\end{equation*}
Thus, we obtain a directed path of the form \eqref{eq:sigma1-1} 
from $y$ to $\mcr{s_{j}x}$. This proves part (1). 

(2) Assume that $\len{x}{y}=n$. By the argument above, we have 
$\len{\mcr{s_{j}x}}{y} \le n-1$. 
Suppose, for a contradiction, that $\len{\mcr{s_{j}x}}{y} < n-1$, 
and take a directed path
\begin{equation*}
\mcr{s_{j}x}=z_{0} \stackrel{\gamma_{1}}{\leftarrow} z_{1}
\stackrel{\gamma_{2}}{\leftarrow} z_{2} 
\stackrel{\gamma_{3}}{\leftarrow} \cdots 
\stackrel{\gamma_{l}}{\leftarrow} z_{l}=y
\end{equation*}
from $y$ to $\mcr{s_{j}x}$ 
whose length $l$ is less than $n-1$. 
Let us show that $x \stackrel{\gamma}{\leftarrow} \mcr{s_{j}x}$ 
for some $\gamma \in \Delta_{0}^{+} \setminus \Delta_{J}^{+}$. 
Assume first that $j \in I_{0}$. 
Since $\pair{x\Lambda}{\alpha_{j}^{\vee}} < 0$ by the assumption, 
we have $x^{-1}\alpha_{j} \in \Delta_{0}^{-} \setminus \Delta_{J}^{-}$, 
and hence $\ell(x)=\ell(s_{j}x)+1$. Also, 
since $x \in W_{0}^{J}$, it follows from Lemma~\ref{lem:mcr} that
$s_{j}x \in W_{0}^{J}$. 
Therefore, if we set $\gamma:=
x^{-1}s_{j}\alpha_{j}=-x^{-1}\alpha_{j} 
\in \Delta_{0}^{+} \setminus \Delta_{J}^{+}$, 
then we obtain $x \stackrel{\gamma}{\leftarrow} s_{j}x=\mcr{s_{j}x}$. 
Assume next that $j=0$. 
Since $\pair{x\Lambda}{-\theta^{\vee}}=
\pair{x\Lambda}{\alpha_{0}^{\vee}} < 0$ 
by the assumption, we have 
$x^{-1}\theta \in \Delta_{0}^{+} \setminus \Delta_{J}^{+}$. 
Define an element $v \in W_{J}$ by 
$r_{\theta}x=\mcr{r_{\theta}x}v$. 
Then we see that $\gamma:=vx^{-1}\theta$ 
is contained in $\Delta_{0}^{+} \setminus \Delta_{J}^{+}$, 
and that 
\begin{equation*}
\mcr{\mcr{s_0x}r_{\gamma}}=
\mcr{\mcr{r_{\theta}x}r_{\gamma}}=
\mcr{r_{\theta}xv^{-1}r_{vx^{-1}\theta}}=
\mcr{r_{\theta}xv^{-1}vx^{-1}r_{\theta}xv^{-1}}=
\mcr{xv^{-1}}=x
\end{equation*}
since $x \in W_{0}^{J}$ and $v \in W_{J}$. 
Also, note that 
$\mcr{s_0x}^{-1}\theta=
\mcr{r_{\theta}x}^{-1}\theta=
 vx^{-1}r_{\theta}\theta=-\gamma \in 
\Delta_{0}^{-} \setminus \Delta_{J}^{-}$. 
Therefore, we deduce from Lemma~\ref{lem:theta} that
\begin{equation*}
x=\mcr{\mcr{s_0 x}r_{\gamma}}
\stackrel{\gamma}{\leftarrow} 
\mcr{r_{\theta}x}=\mcr{s_{0}x}.
\end{equation*}
Thus, we obtain a directed path
\begin{equation*}
x \stackrel{\gamma}{\leftarrow} 
\mcr{s_{j}x}=z_{0} \stackrel{\gamma_{1}}{\leftarrow} z_{1}
\stackrel{\gamma_{2}}{\leftarrow} z_{2} 
\stackrel{\gamma_{3}}{\leftarrow} \cdots 
\stackrel{\gamma_{l}}{\leftarrow} z_{l}=y
\end{equation*}
from $y$ to $x$ whose length is $l+1 < n=\len{x}{y}$. 
This contradicts the definition of $\len{x}{y}$. 
This proves part (2). 

(3) We should remark that 
$\pair{\Lambda}{z_{k}\beta_{k}^{\vee}}=\pair{\Lambda}{\beta_{k}^{\vee}}$ 
for each $1 \le k \le p-1$, since $z_{k} \in W_{J}$. Hence 
the assertion of part (3) follows immediately from 
the definition of a directed $\sigma$-path. 
This completes the proof of Lemma~\ref{lem:sigma1}. 
\end{proof}

The following lemma can be shown in the same way as 
Lemma~\ref{lem:sigma1}. If $j \in I_{0}$, then 
use Lemma~\ref{lem:dist1}\,(1) and (4) instead of 
Lemma~\ref{lem:dist1}\,(2) and (3), respectively; if $j=0$, then 
use Lemma~\ref{lem:dist2}\,(1) and (4) instead of 
Lemma~\ref{lem:dist2}\,(2) and (3), respectively. 
%
%
\begin{lem} \label{lem:sigma2}
Keep the notation and setting in Lemma~\ref{lem:sigma1}. 
\begin{enu}
\item If there exists $1 \le p \le n$ such that 
$\pair{w_{k}\Lambda}{\alpha_{j}^{\vee}} > 0$ for all $p \le k \le n$ and 
$\pair{w_{p-1}\Lambda}{\alpha_{j}^{\vee}} \le 0$, 
then $w_{p-1}=\mcr{s_{j}w_{p}}$, 
and there exists a directed path from $\mcr{s_{j}y}$ 
to $x$ of the form: 
%
%
\begin{equation} \label{eq:sigma2-1}
x=w_{0} \stackrel{\beta_{1}}{\leftarrow} \cdots 
\stackrel{\beta_{p-1}}{\leftarrow}
w_{p-1}=\mcr{s_{j}w_{p}}
\stackrel{z_{p+1}\beta_{p+1}}{\leftarrow} \cdots 
\stackrel{z_{n}\beta_{n}}{\leftarrow} 
\mcr{s_{j}w_{n}}=\mcr{s_{j}y}.
\end{equation}
Here, if $j \in I_{0}$, then we define $z_{k} \in W_{J}$ to 
be the identity element for all $p+1 \le k \le n$; 
if $j=0$, then we define $z_{k} \in W_{J}$ by 
$r_{\theta}w_{k}=\mcr{r_{\theta}w_{k}}z_{k}$ 
for each $p+1 \le k \le n$. 

\item If the directed path \eqref{eq:sigma1-0} from $y$ to $x$ is 
shortest, i.e., $\len{x}{y}=n$, then the directed path 
\eqref{eq:sigma2-1} from $\mcr{s_{j}y}$ to $x$ 
is also shortest, i.e., $\len{x}{\mcr{s_{j}y}}=n-1$. 

\item If the directed path \eqref{eq:sigma1-0} is a 
directed $\sigma$-path from $y$ to $x$ 
for some rational number $0 < \sigma < 1$, 
then the directed path \eqref{eq:sigma2-1} is a 
directed $\sigma$-path from $\mcr{s_{j}y}$ to $x$.
\end{enu}
\end{lem}
%
%
\begin{lem} \label{lem:stable1}
Let $\eta=(x_{1},\,x_{2},\,\dots,\,x_{s}\,;\,
 \sigma_{0},\,\sigma_{1},\,\dots,\,\sigma_{s})
 \in \ti{\BB}(\lambda)_{\cl}$. 
Let $j \in I$ and $1 \le u \le s-1$ be such that 
$\pair{x_{u+1}\Lambda}{\alpha_{j}^{\vee}} > 0$. Let 
\begin{equation*}
x_{u}=w_{0} \stackrel{\beta_{1}}{\leftarrow} w_{1}
\stackrel{\beta_{2}}{\leftarrow} w_{2} 
\stackrel{\beta_{3}}{\leftarrow} \cdots 
\stackrel{\beta_{n}}{\leftarrow} w_{n}=x_{u+1}
\end{equation*}
be a directed $\sigma_{u}$-path from $x_{u+1}$ to $x_{u}$. 
If there exists $0 \le k < n$ 
such that $\pair{w_{k}\Lambda}{\alpha_{j}^{\vee}} \le 0$, 
then $H^{\eta}_{j}(\sigma_{u}) \in \BZ$. In particular, 
if $\pair{x_{u}\Lambda}{\alpha_{j}^{\vee}} \le 0$, then 
$H^{\eta}_{j}(\sigma_{u}) \in \BZ$. 
\end{lem}

\begin{proof}
We see from the definition 
that $\eta':=(x_{1},\,x_{2},\,\dots,\,x_{u},\,x_{u+1}\,;\,
\sigma_{0},\,\sigma_{1},\,\dots,\,\sigma_{u},\,\sigma_{s})$ 
is an element of $\ti{\BB}(\lambda)_{\cl}$. 
Also, observe that 
$\eta'(t)=\eta(t)$ for $0 \le t \le \sigma_{u+1}$, 
and hence $H^{\eta'}_{j}(t)=H^{\eta}_{j}(t)$ for 
$0 \le t \le \sigma_{u+1}$. It follows that
\begin{equation*}
H^{\eta}_{j}(\sigma_{u})=H^{\eta'}_{j}(\sigma_{u})=
 H^{\eta'}_{j}(1)-(1-\sigma_{u})\pair{x_{u+1}\Lambda}{\alpha_{j}^{\vee}}.
\end{equation*}
Since $\eta'(1) \in P_{\cl}$ (and hence $H^{\eta'}_{j}(1) \in \BZ$) 
by Lemma~\ref{lem:weight},
it suffices to show that 
$(1-\sigma_{u})\pair{x_{u+1}\Lambda}{\alpha_{j}^{\vee}} \in \BZ$. 

We deduce from Lemma~\ref{lem:sigma2} that 
there exists a directed $\sigma_{u}$-path 
from $\mcr{s_{j}x_{u+1}}$ to $x_{u}$. 
Therefore, 
$\eta''=
(x_{1},\,x_{2},\,\dots,\,x_{u},\,\mcr{s_{j}x_{u+1}}\,;\,
\sigma_{0},\,\sigma_{1},\,\dots,\,\sigma_{u},\,\sigma_{s})$
is also an element of $\ti{\BB}(\lambda)_{\cl}$. 
Since both $\eta'(1)$ and $\eta''(1)$ are 
contained in $\Lambda+Q_{0}$ by Lemma~\ref{lem:weight}, 
we have $\eta'(1)-\eta''(1) \in Q_{0}$. 
Also, we have
\begin{align*}
(Q_{0} \ni) \ 
\eta'(1)-\eta''(1) & = 
(1-\sigma_{u})x_{u+1}\Lambda-(1-\sigma_{u})s_{j}x_{u+1}\Lambda \\[3mm]
& = 
\begin{cases}
(1-\sigma_{u})\pair{x_{u+1}\Lambda}{\alpha_{j}^{\vee}}\alpha_{j} 
  & \text{if $j \in I_{0}$}, \\[1.5mm]
(1-\sigma_{u})\pair{x_{u+1}\Lambda}{\alpha_{j}^{\vee}}(-\theta)
  & \text{if $j=0$}. 
\end{cases}
\end{align*}
Here we remark that $\theta=\delta-\alpha_{0}=
\sum_{j \in I_{0}}a_{j}\alpha_{j}$, 
and the greatest common divisor of 
the $a_{j}$, $j \in I_{0}$, is equal to $1$. 
From these, we conclude that 
$(1-\sigma_{u})\pair{x_{u+1}\Lambda}{\alpha_{j}^{\vee}} \in \BZ$, 
thereby completing the proof of the proposition. 
\end{proof}

The following lemma can be shown in the same way as 
Lemma~\ref{lem:stable1}; noting that 
$\pi':=(x_{u},\,x_{u+1}\,\dots,\,x_{s};\,
\sigma_{0},\,\sigma_{u},\,\sigma_{u+1},\,\dots,\,\sigma_{s})$ 
is an element of $\ti{\BB}(\lambda)_{\cl}$, 
use $\pi'$ instead of $\eta'$ and the fact that 
$H^{\pi'}_{j}(1)-H^{\pi'}_{j}(1-t)=
 H^{\eta}_{j}(1)-H^{\eta}_{j}(1-t)$ for 
$0 \le t \le 1-\sigma_{u-1}$. 
%
%
\begin{lem} \label{lem:stable1a}
Let $\eta=(x_{1},\,x_{2},\,\dots,\,x_{s}\,;\,
 \sigma_{0},\,\sigma_{1},\,\dots,\,\sigma_{s})
 \in \ti{\BB}(\lambda)_{\cl}$. 
Let $j \in I$ and $1 \le u \le s-1$ be such that 
$\pair{x_{u}\Lambda}{\alpha_{j}^{\vee}} < 0$. Let 
\begin{equation*}
x_{u}=w_{0} \stackrel{\beta_{1}}{\leftarrow} w_{1}
\stackrel{\beta_{2}}{\leftarrow} w_{2} 
\stackrel{\beta_{3}}{\leftarrow} \cdots 
\stackrel{\beta_{n}}{\leftarrow} w_{n}=x_{u+1}
\end{equation*}
be a directed $\sigma_{u}$-path from $x_{u+1}$ to $x_{u}$. 
If there exists $0 < k \le n$ 
such that $\pair{w_{k}\Lambda}{\alpha_{j}^{\vee}} \ge 0$, 
then $H^{\eta}_{j}(\sigma_{u}) \in \BZ$. In particular, 
if $\pair{x_{u+1}\Lambda}{\alpha_{j}^{\vee}} \ge 0$, then 
$H^{\eta}_{j}(\sigma_{u}) \in \BZ$. 
\end{lem}
%
%
\begin{prop} \label{prop:ip}
Let $\lambda \in \sum_{i \in I_{0}} \BZ_{\ge 0} \vpi_{i}$ be as above. 
Both $\ti{\BB}(\lambda)_{\cl}$ and $\ha{\BB}(\lambda)_{\cl}$ are 
contained in $\BP_{\cl,\,\INT}$ under the identification \eqref{eq:QBG-path} of 
a rational path with a piecewise-linear, continuous map.
\end{prop}

\begin{proof}
Since $\ha{\BB}(\lambda)_{\cl} \subset \ti{\BB}(\lambda)_{\cl}$ 
by the definitions, it suffices to show that 
$\ti{\BB}(\lambda)_{\cl} \subset \BP_{\cl,\,\INT}$. 
Let $\eta=(x_{1},\,x_{2},\,\dots,\,x_{s}\,;\,
 \sigma_{0},\,\sigma_{1},\,\dots,\,\sigma_{s})
 \in \ti{\BB}(\lambda)_{\cl}$. 
We have shown that $\eta(1) \in P_{\cl}$ for every 
$\eta \in \ti{\BB}(\lambda)_{\cl}$ (see Lemma~\ref{lem:weight}). 
It remains to show that for every $j \in I$, 
all local minima of the function $H_{j}^{\eta}(t)$ 
are integers. Fix $j \in I$, and assume that 
the function $H_{j}^{\eta}(t)$ attains 
a local minimum at $t' \in [0,1]$;
we may assume that $t'=\sigma_{u}$ for some $0 \le u \le s$. 
If $u=0$ (resp., $u=s$), then 
$H_{j}^{\eta}(t')=H_{j}^{\eta}(0)=0 \in \BZ$ 
(resp., $H_{j}^{\eta}(t')=H_{j}^{\eta}(1) \in \BZ$) 
since $\eta(0)=0$ (resp., $\eta(1) \in P_{\cl}$). 
If $0 < u  < s$, then we have either 
$\pair{x_{u}\Lambda}{\alpha_{j}^{\vee}} \le 0$ and 
$\pair{x_{u+1}\Lambda}{\alpha_{j}^{\vee}} > 0$, or 
$\pair{x_{u}\Lambda}{\alpha_{j}^{\vee}} < 0$ and 
$\pair{x_{u+1}\Lambda}{\alpha_{j}^{\vee}} \ge 0$.
Therefore, it follows from Lemma~\ref{lem:stable1} 
or \ref{lem:stable1a} that 
$H^{\eta}_{j}(\sigma_{u}) \in \BZ$. 
Thus, we have proved the proposition. 
\end{proof}
%
%
\begin{lem} \label{lem:stable2}
Let $\eta=(x_{1},\,x_{2},\,\dots,\,x_{s}\,;\,
 \sigma_{0},\,\sigma_{1},\,\dots,\,\sigma_{s})
 \in \ti{\BB}(\lambda)_{\cl}$. 
Let $j \in I$ and $1 \le u \le s-1$ be such that 
$\pair{x_{u+1}\Lambda}{\alpha_{j}^{\vee}} > 0$ and 
$H^{\eta}_{j}(\sigma_{u}) \not\in \BZ$. Let 
%
%
\begin{equation} \label{eq:stable2-0}
x_{u}=w_{0} \stackrel{\beta_{1}}{\leftarrow} w_{1}
\stackrel{\beta_{2}}{\leftarrow} w_{2} 
\stackrel{\beta_{3}}{\leftarrow} \cdots 
\stackrel{\beta_{n}}{\leftarrow} w_{n}=x_{u+1}
\end{equation}
be a directed $\sigma_{u}$-path from $x_{u+1}$ to $x_{u}$.
Then, $\pair{w_{k}\Lambda}{\alpha_{j}^{\vee}} > 0$ for all $0 \le k \le n$, 
and there exists a directed $\sigma_{u}$-path from 
$\mcr{s_{j}x_{u+1}}$ to $\mcr{s_{j}x_{u}}$ of the form: 
%
%
\begin{equation} \label{eq:stable2-1}
\mcr{s_{j}x_{u}}=\mcr{s_{j}w_{0}} 
\stackrel{z_{1}\beta_{1}}{\leftarrow} \mcr{s_{j}w_{1}}
\stackrel{z_{2}\beta_{2}}{\leftarrow} \cdots 
\stackrel{z_{n}\beta_{n}}{\leftarrow} \mcr{s_{j}w_{n}}=
\mcr{s_{j}x_{u+1}}. 
\end{equation}
Here, if $j \in I_{0}$, then we define $z_{k} \in W_{J}$ to be 
the identity element for all $1 \le k \le n$; 
if $j=0$, then we define $z_{k} \in W_{J}$ by
$r_{\theta}w_{k}=\mcr{r_{\theta}w_{k}}z_{k}$ 
for each $1 \le k \le n$. 
Moreover, if \eqref{eq:stable2-0} is a shortest 
directed path from $x_{u+1}$ to $x_{u}$, i.e., 
$\len{x_{u}}{x_{u+1}}=n$, then 
\eqref{eq:stable2-1} is a shortest 
directed path from $\mcr{s_{j}x_{u+1}}$ to $\mcr{s_{j}x_{u}}$, i.e., 
$\len{\mcr{s_{j}x_{u}}}{\mcr{s_{j}x_{u+1}}}=n$. 
\end{lem}
\begin{proof}
It follows from Lemma~\ref{lem:stable1} that 
if $H^{\eta}_{j}(\sigma_{u}) \not\in \BZ$, 
then $\pair{w_{k}\Lambda}{\alpha_{j}^{\vee}} > 0$ for all $0 \le k \le n$ 
(in particular, $\pair{x_{u}\Lambda}{\alpha_{j}^{\vee}} > 0$). 
Assume that $j \in I_{0}$ (resp., $j=0$), and suppose, 
for a contradiction, that 
$w_{k}\beta_{k} = \pm \alpha_{j}$ (resp., $= \pm \theta$)
for some $1 \le k \le n$. 
Then, $w_{k-1}\Lambda=w_{k}r_{\beta_{k}}\Lambda=
r_{w_{k}\beta_{k}}w_{k}\Lambda=s_{j}w_{k}\Lambda$, 
and hence $\pair{w_{k-1}\Lambda}{\alpha_{j}^{\vee}}=
\pair{s_{j}w_{k}\Lambda}{\alpha_{j}^{\vee}}=-\pair{w_{k}\Lambda}{\alpha_{j}^{\vee}}$, 
which contradicts the fact that 
$\pair{w_{k-1}\Lambda}{\alpha_{j}^{\vee}} > 0$ and 
$\pair{w_{k}\Lambda}{\alpha_{j}^{\vee}} > 0$. 
Thus, we conclude that 
$w_{k}\beta_{k} \ne \pm \alpha_{j}$ (resp., $\ne \pm \theta$)
for any $1 \le k \le n$. 
Therefore, we deduce from Lemma~\ref{lem:dist1}\,(1) 
(resp., Lemma~\ref{lem:dist2}\,(1)) that 
there exists a directed path of the form \eqref{eq:stable2-1} 
from $\mcr{s_{j}x_{u+1}}$ to $\mcr{s_{j}x_{u}}$. 
Because the directed path \eqref{eq:stable2-0} is a directed $\sigma_{u}$-path, 
we have $\sigma_{u}\pair{\Lambda}{\beta_{k}^{\vee}} \in \BZ$. 
Also, it follows immediately that 
$\sigma_{u}\pair{\Lambda}{z\beta_{k}^{\vee}}=
 \sigma_{u}\pair{\Lambda}{\beta_{k}^{\vee}} \in \BZ$
since $z \in W_{J}$.
Thus, the directed path \eqref{eq:stable2-1} is 
a directed $\sigma_{u}$-path from 
$\mcr{s_{j}x_{u+1}}$ to $\mcr{s_{j}x_{u}}$. 

Now, we assume that $\len{x_{u}}{x_{u+1}}=n$, and 
suppose, for a contradiction, that there exists a directed path 
%
%
\begin{equation} \label{eq:dp11}
\mcr{s_{j}x_{u}}=z_{0} 
\stackrel{\gamma_{1}}{\leftarrow} z_{1}
\stackrel{\gamma_{2}}{\leftarrow} z_{2} 
\stackrel{\gamma_{3}}{\leftarrow} \cdots 
\stackrel{\gamma_{l}}{\leftarrow} z_{l}=\mcr{s_{j}x_{u+1}}
\end{equation}
from $\mcr{s_{j}x_{u+1}}$ to $\mcr{s_{j}x_{u}}$ 
whose length $l$ is less than $n$. 
Let us show that 
$\mcr{s_{j}x_{u+1}} \stackrel{\gamma}{\leftarrow} x_{u+1}$ 
for some $\gamma \in \Delta_{0}^{+} \setminus \Delta_{J}^{+}$. 
Assume first that $j \in I_{0}$. Since $\pair{x_{u+1}\Lambda}{\alpha_{j}^{\vee}} > 0$, 
we have $\gamma:=x_{u+1}^{-1}\alpha_{j} 
\in \Delta_{0}^{+} \setminus \Delta_{J}^{+}$, 
and hence $\ell(s_{j}x_{u+1})=\ell(x_{u+1})+1$. Also, by Lemma~\ref{lem:mcr}, 
$s_{j}x_{u+1} \in W_{0}^{J}$. Since $s_{j}x_{u+1}=x_{u+1}r_{\gamma}$, we obtain
$\mcr{s_{j}x_{u+1}}=s_{j}x_{u+1} \stackrel{\gamma}{\leftarrow} x_{u+1}$. 
Assume next that $j=0$. 
Since $\pair{x_{u+1}\Lambda}{\theta^{\vee}}=-\pair{x_{u+1}\Lambda}{\alpha_{0}^{\vee}} < 0$ 
by the assumption, it follows that 
$x_{u+1}^{-1}\theta \in \Delta_{0}^{-} \setminus \Delta_{J}^{-}$. 
Therefore, if we set $\gamma:=-x_{u+1}^{-1}\theta \in 
\Delta_{0}^{+} \setminus \Delta_{J}^{+}$, 
then $s_{0}x_{u+1}=r_{\theta}x_{u+1}=x_{u+1}r_{\gamma}$, and 
%
%
%
we obtain 
$\mcr{s_{0}x_{u+1}} \stackrel{\gamma}{\leftarrow} x_{u+1}$ 
by Lemma~\ref{lem:theta}. 
By concatenating the directed path \eqref{eq:dp11} and 
$\mcr{s_{j}x_{u+1}} \stackrel{\gamma}{\leftarrow} x_{u+1}$, 
we obtain a directed path 
from $x_{u+1}$ to $\mcr{s_{j}x_{u}}$ whose length is $l+1$. 
Since $\pair{x_{u+1}\Lambda}{\alpha_{j}^{\vee}} > 0$ and 
$\pair{s_{j}x_{u}\Lambda}{\alpha_{j}^{\vee}}=
 -\pair{x_{u}\Lambda}{\alpha_{j}^{\vee}} < 0$, 
we deduce from Lemma~\ref{lem:sigma1}\,(1) that 
there exists a directed path from $x_{u+1}$ to 
$\mcr{s_{j}\mcr{s_{j}x_{u}}}=x_{u}$ whose length is $(l+1)-1=l$. 
However, this contradicts the fact that 
$n=\len{x_{u}}{x_{u+1}}$ since $l < n$. 
This proves the lemma. 
\end{proof}
%
%
\subsection{Explicit description of the image of a quantum LS path under
the action of root operators.}
\label{subsec:ro}
In the course of the proof of the following proposition, 
we obtain an explicit description of 
the image of a quantum LS path as a rational path under
the action of root operators; 
this description is similar to the one given in \cite{Lit-I}. 
%
%
\begin{prop} \label{prop:stable}
Both of the sets $\ti{\BB}(\lambda) \cup \{\bzero\}$ and 
$\ha{\BB}(\lambda) \cup \{\bzero\}$ are stable 
under the action of the root operators $f_{j}$ for all $j \in I$. 
\end{prop}

\begin{proof}
Fix $j \in I$. 
Let $\eta=(x_{1},\,x_{2},\,\dots,\,x_{s}\,;\,
\sigma_{0},\,\sigma_{1},\,\dots,\,\sigma_{s}) \in 
\ti{\BB}(\lambda)_{\cl}$, and assume that 
$f_{j}\eta \ne \bzero$. 
It follows that the point 
$t_{0}=\max\bigl\{t \in [0,1] \mid 
       H^{\eta}_{j}(t)=m^{\eta}_{j}\bigr\}$
is equal to $\sigma_{u}$ for some $0 \le u < s$. 
Let $u \le m < s$ be such that 
$\sigma_{m} < t_{1} \le \sigma_{m+1}$; recall that 
$t_{1}=\min\bigl\{t \in [t_{0},1] \mid
       H^{\eta}_{j}(t)=m^{\eta}_{j}+1\bigr\}$. 
Note that 
the function $H^{\eta}_{j}(t)$ 
is strictly increasing on $[t_{0},\,t_{1}]$, 
which implies that $\pair{x_{p}\Lambda}{\alpha_{j}^{\vee}} > 0$ 
for all $u+1 \le p \le m+1$. 

\paragraph{Case 1.} Assume that 
$x_{u} \ne \mcr{s_{j}x_{u+1}}$ or $u=0$, 
and that $\sigma_{m} < t_{1} < \sigma_{m+1}$.
Then we deduce from the definition of the root operator $f_{j}$ 
(for the case $j=0$, see also Remark~\ref{rem:ro_f0}; 
 cf. \cite[Proposition~4.7\,a)]{Lit-A}) that 
\begin{align*}
& f_{j}\eta=
(x_{1},\,x_{2},\,\dots,\,x_{u},\,\mcr{s_{j}x_{u+1}},\,\dots,\,
 \mcr{s_{j}x_{m}},\,\mcr{s_{j}x_{m+1}},\,x_{m+1},\,x_{m+2},\,\dots,\,x_{s}\,;\, \\
& \hspace*{70mm}
\sigma_{0},\,\sigma_{1},\,\dots,\,\sigma_{u},\,\dots,\,
\sigma_{m},\,t_{1},\,\sigma_{m+1},\,\dots,\,\sigma_{s});
\end{align*}
note that $\mcr{s_{j}x_{p}} \ne \mcr{s_{j}x_{p+1}}$
for all $u+1 \le p \le m$, and that 
$\mcr{s_{j}x_{m+1}} \ne x_{m+1}$ since 
$\pair{x_{m+1}\Lambda}{\alpha_{j}^{\vee}} > 0$ as mentioned above. 
In order to prove that $f_{j}\eta \in \ti{\BB}(\lambda)_{\cl}$, 
we need to verify that 
\begin{enumerate}
\renewcommand{\labelenumi}{(\roman{enumi})}

\item there exists a directed $\sigma_{u}$-path 
from $\mcr{s_{j}x_{u+1}}$ to $x_{u}$ (when $u > 0$); 

\item there exists a directed $\sigma_{p}$-path 
from $\mcr{s_{j}x_{p+1}}$ to $\mcr{s_{j}x_{p}}$ 
for each $u+1 \le p \le m$; 

\item there exists a directed $t_{1}$-path 
from $x_{m+1}$ to $\mcr{s_{j}x_{m+1}}$. 

\end{enumerate}
Also, we will show that 
if $\eta \in \ha{\BB}(\lambda)_{\cl}$, then 
the directed paths in (i)--(iii) above can be chosen 
from the shortest ones, which implies that 
$f_{j}\eta \in \ha{\BB}(\lambda)_{\cl}$. 

(i) We deduce from the definition of $t_{0}=\sigma_{u}$ that 
$\pair{x_{u}\Lambda}{\alpha_{j}^{\vee}} \le 0$ and 
$\pair{x_{u+1}\Lambda}{\alpha_{j}^{\vee}} > 0$. 
Since $\eta \in \ti{\BB}(\lambda)_{\cl}$, 
there exists a directed $\sigma_{u}$-path 
from $x_{u+1}$ to $x_{u}$. Hence it follows from 
Lemma~\ref{lem:sigma2}\,(1),\,(3) that there exists 
a directed $\sigma_{u}$-path 
from $\mcr{s_{j}x_{u+1}}$ to $x_{u}$. 
Furthermore, we see from the definition of 
$\ha{\BB}(\lambda)_{\cl}$ and 
Lemma~\ref{lem:sigma2}\,(2) that 
if $\eta \in \ha{\BB}(\lambda)_{\cl}$, then 
there exists a directed $\sigma_{u}$-path 
from $\mcr{s_{j}x_{u+1}}$ to 
$x_{u}$ whose length is equal to 
$\len{x_{u}}{\mcr{s_{j}x_{u+1}}}$. 

(ii) Recall that $H^{\eta}_{j}(t)$ is strictly increasing 
on $[t_{0},t_{1}]$, and that 
$H^{\eta}_{j}(t_{0})=m^{\eta}_{j}$ and 
$H^{\eta}_{j}(t_{1})=m^{\eta}_{j}+1$. 
Hence it follows that 
$H^{\eta}_{j}(\sigma_{p}) \notin \BZ$ 
for all $u+1 \le p \le m$. 
Therefore, we deduce from Lemma~\ref{lem:stable2} that 
there exists a directed $\sigma_{p}$-path 
from $\mcr{s_{j}x_{p+1}}$ to $\mcr{s_{j}x_{p}}$ 
for each $u+1 \le p \le m$. 
Furthermore, we see from the definition of 
$\ha{\BB}(\lambda)_{\cl}$ and 
Lemma~\ref{lem:stable2} that 
if $\eta \in \ha{\BB}(\lambda)_{\cl}$, then 
for each $u+1 \le p \le m$, 
there exists a directed $\sigma_{p}$-path 
from $\mcr{s_{j}x_{p+1}}$ to $\mcr{s_{j}x_{p}}$ 
whose length is equal to 
$\len{\mcr{s_{j}x_{p}}}{\mcr{s_{j}x_{p+1}}}$. 

(iii) Since $\pair{x_{m+1}\Lambda}{\alpha_{j}^{\vee}} > 0$, 
by the same argument as 
in the second paragraph of the proof of Lemma~\ref{lem:stable2}, 
we obtain $\mcr{s_{j}x_{m+1}} \stackrel{\gamma}{\leftarrow} x_{m+1}$, 
with
\begin{equation*}
\gamma:=
 \begin{cases}
 x_{m+1}^{-1}\alpha_{j} & \text{if $j \in I_{0}$}, \\[1.5mm]
 x_{m+1}^{-1}(-\theta) & \text{if $j=0$}; 
 \end{cases}
\end{equation*}
note that the directed path 
$\mcr{s_{j}x_{m+1}} \stackrel{\gamma}{\leftarrow} x_{m+1}$ 
is obviously shortest since its length is equal to $1$. 
Let us show that $t_{1}\pair{\Lambda}{\gamma^{\vee}} \in \BZ$. 
It is easily checked that $\pair{\Lambda}{\gamma^{\vee}}=
\pair{x_{m+1}\Lambda}{\alpha_{j}^{\vee}}$. Also, we have
$\eta(t_{1})=t_{1}x_{m+1}\Lambda+
 \sum_{k=1}^{m}\sigma_{k}(x_{k}\Lambda-x_{k+1}\Lambda)$, and hence
\begin{equation*}
\BZ \ni m^{\eta}_{j}+1=
H^{\eta}_{j}(t_{1})=t_{1}\pair{x_{m+1}\Lambda}{\alpha_{j}^{\vee}}+
 \sum_{k=1}^{m}\pair{\sigma_{k}(x_{k}\Lambda-x_{k+1}\Lambda)}{\alpha_{j}^{\vee}}.
\end{equation*}
Since $\sigma_{k}(x_{k}\Lambda-x_{k+1}\Lambda) \in Q_{0}$ 
for each $1 \le k \le m$ by Lemma~\ref{lem:sigma0}, 
it follows from the equation above 
that $t_{1}\pair{x_{m+1}\Lambda}{\alpha_{j}^{\vee}} \in \BZ$, 
and hence $t_{1}\pair{\Lambda}{\gamma^{\vee}} \in \BZ$. 
Thus, we have verified that there exists a directed $t_{1}$-path 
from $x_{m+1}$ to $\mcr{s_{j}x_{m+1}}$ whose length is equal to 
$\len{\mcr{s_{j}x_{m+1}}}{x_{m+1}}=1$. 

Combining these, we conclude that 
$f_{j}\eta$ is an element of $\ti{\BB}(\lambda)_{\cl}$, and that 
if $\eta \in \ha{\BB}(\lambda)_{\cl}$, then 
$f_{j}\eta \in \ha{\BB}(\lambda)_{\cl}$. 

\paragraph{Case 2.} Assume that 
$x_{u} \ne \mcr{s_{j}x_{u+1}}$ or $u=0$, and that $t_{1}=\sigma_{m+1}$.
Then we deduce from the definition of the root operator $f_{j}$ 
(for the case $j=0$, see also Remark~\ref{rem:ro_f0}; 
 cf. \cite[Proposition~4.7\,a) and Remark~4.8]{Lit-A}) that 
\begin{align*}
& f_{j}\eta=
(x_{1},\,x_{2},\,\dots,\,x_{u},\,\mcr{s_{j}x_{u+1}},\,\dots,\,
 \mcr{s_{j}x_{m}},\,\mcr{s_{j}x_{m+1}},\,x_{m+2},\,\dots,\,x_{s}\,;\, \\
& \hspace*{70mm}
\sigma_{0},\,\sigma_{1},\,\dots,\,\sigma_{u},\,\dots,\,
\sigma_{m},\,t_{1}=\sigma_{m+1},\,\dots,\,\sigma_{s}).
\end{align*}
First, we observe that $\pair{x_{m+2}\Lambda}{\alpha_{j}^{\vee}} \ge 0$. 
Indeed, suppose, contrary to our claim, that 
$\pair{x_{m+2}\Lambda}{\alpha_{j}^{\vee}} < 0$. 
Since $H^{\eta}_{j}(\sigma_{m+1})=H^{\eta}_{j}(t_{1})=m^{\eta}_{j}+1$, 
it follows immediately that $H^{\eta}_{j}(\sigma_{m+1}+\epsilon) < 
m^{\eta}_{j}+1$ for sufficiently small $\epsilon > 0$, 
and hence the minimum $M$ of the function $H^{\eta}_{j}(t)$ on 
$[t_{1},\,1]$ is (strictly) less than $m^{\eta}_{j}+1$. 
Here we recall from Proposition~\ref{prop:ip} 
that all local minima of the function 
$H^{\eta}_{j}(t)$ are integers. 
Hence we deduce that $M=m^{\eta}_{j}$, 
which contradicts the definition of $t_{0}$. 
Thus, we obtain $\pair{x_{m+2}\Lambda}{\alpha_{j}^{\vee}} \ge 0$. 
Since $\pair{x_{m+1}\Lambda}{\alpha_{j}^{\vee}} > 0$, and hence
$\pair{s_{j}x_{m+1}\Lambda}{\alpha_{j}^{\vee}} < 0$, it follows that
$\mcr{s_{j}x_{m+1}} \ne x_{m+2}$. 

Now, in order to prove that 
$f_{j}\eta \in \ti{\BB}(\lambda)_{\cl}$, 
we need to verify that 
\begin{enumerate}
\renewcommand{\labelenumi}{(\roman{enumi})}

\item there exists a directed $\sigma_{u}$-path from 
$\mcr{s_{j}x_{u+1}}$ to $x_{u}$ (when $u > 0$); 

\item there exists a directed $\sigma_{p}$-path from 
$\mcr{s_{j}x_{p+1}}$ to $\mcr{s_{j}x_{p}}$ 
for each $u+1 \le p \le m$;

\setcounter{enumi}{3}

\item there exists a directed $\sigma_{m+1}$-path from 
$x_{m+2}$ to $\mcr{s_{j}x_{m+1}}$ 
(when $m+1 < s$). 

\end{enumerate}
We can verify (i) and (ii) by the same 
argument as for (i) and (ii) in Case 1, respectively. 
Hence it remains to show (iv). Also, in order to prove 
that $\eta \in \ha{\BB}(\lambda)_{\cl}$ implies 
$f_{j}\eta \in \ha{\BB}(\lambda)_{\cl}$, it suffices to check that 
the directed paths in (i), (ii), and (iv) above can be chosen from 
the shortest ones. 
We can show this claim for (i) and (ii) in the same way as 
for (i) and (ii) in Case 1, respectively. 
So, it remains to show it for (iv).

(iv) As in the proof of (iii) in Case 1, it can be shown that 
there exists a directed $t_{1}$-path
(and hence directed $\sigma_{m+1}$-path since $t_{1}=\sigma_{m+1}$ 
by the assumption) from $x_{m+1}$ to $\mcr{s_{j}x_{m+1}}$ 
whose length is equal to $1$. 
Also, it follows from the definition that 
there exists a directed $\sigma_{m+1}$-path from $x_{m+2}$ to $x_{m+1}$. 
Concatenating these directed $\sigma_{m+1}$-paths, we obtain 
a directed $\sigma_{m+1}$-path from $x_{m+2}$ to $\mcr{s_{j}x_{m+1}}$. 
Thus, we have proved that $f_{j}\eta \in \ti{\BB}(\lambda)_{\cl}$. 

Assume now that $\eta \in \ha{\BB}(\lambda)_{\cl}$, and set 
$n:=\len{x_{m+1}}{x_{m+2}}$. 
We see from the argument above that 
there exists a directed $\sigma_{m+1}$-path
from $x_{m+2}$ to $\mcr{s_{j}x_{m+1}}$ 
whose length is equal to $n+1$. 
Suppose, for a contradiction, that there exists a directed path from 
$x_{m+2}$ to $\mcr{s_{j}x_{m+1}}$ 
whose length $l$ is less than $n+1$.
Since $\pair{s_{j}x_{m+1}\Lambda}{\alpha_{j}^{\vee}} < 0$ and 
$\pair{x_{m+2}\Lambda}{\alpha_{j}^{\vee}} \ge 0$ as seen above, 
we deduce from Lemma~\ref{lem:sigma1} 
that there exists a directed path from $x_{m+2}$ to 
$\mcr{s_{j}\mcr{s_{j}x_{m+1}}}=\mcr{x_{m+1}}=x_{m+1}$ 
whose length is equal to $l-1 < n$, which contradicts 
$n=\len{x_{m+1}}{x_{m+2}}$.
Thus, we have proved that if $\eta \in \ha{\BB}(\lambda)_{\cl}$, 
then $f_{j}\eta \in \ha{\BB}(\lambda)_{\cl}$.

\paragraph{Case 3.} Assume that 
$x_{u} = \mcr{s_{j}x_{u+1}}$ and $\sigma_{m} < t_{1} < \sigma_{m+1}$.
Then we deduce from the definition of the root operator $f_{j}$ 
(for the case $j=0$, see also Remark~\ref{rem:ro_f0}; 
 cf. \cite[Proposition~4.7\,a) and Remark~4.8]{Lit-A}) that 
\begin{align*}
& f_{j}\eta=
(x_{1},\,x_{2},\,\dots,\,x_{u}=\mcr{s_{j}x_{u+1}},\,
 \mcr{s_{j}x_{u+2}},\,\dots,\, \\
& \hspace*{50mm}
 \mcr{s_{j}x_{m}},\,\mcr{s_{j}x_{m+1}},\,x_{m+1},\,x_{m+2},\,\dots,\,x_{s}\,;\, \\
& \hspace*{50mm}
\sigma_{0},\,\sigma_{1},\,\dots,\,\sigma_{u-1},\,\sigma_{u+1},\,\dots,\,
\sigma_{m},\,t_{1},\,\sigma_{m+1},\,\dots,\,\sigma_{s});
\end{align*}
note that $\mcr{s_{j}x_{m+1}} \ne x_{m+1}$ since 
$\pair{x_{m+1}\Lambda}{\alpha_{j}^{\vee}} > 0$. 
In order to prove that $f_{j}\eta \in \ti{\BB}(\lambda)_{\cl}$, 
we need to verify that 

\begin{enumerate}
\renewcommand{\labelenumi}{(\roman{enumi})}

\setcounter{enumi}{1}

\item there exists a directed $\sigma_{p}$-path from 
$\mcr{s_{j}x_{p+1}}$ to $\mcr{s_{j}x_{p}}$ 
for each $u+1 \le p \le m$; 

\item there exists a directed $t_{1}$-path 
from $x_{m+1}$ to $\mcr{s_{j}x_{m+1}}$. 

\end{enumerate}
We can verify (ii) and (iii) 
by the same argument as for (ii) and (iii) in Case 1, respectively. 
Also, in the same way as in the proofs of (ii) and (iii) in Case 1, 
respectively, we can check that if $\eta \in \ha{\BB}(\lambda)_{\cl}$, then 
the directed paths in (ii) and (iii) above can be chosen from the shortest ones. 
Thus we have proved that $f_{j}\eta \in \ti{\BB}(\lambda)_{\cl}$, and 
that $\eta \in \ha{\BB}(\lambda)_{\cl}$ implies 
$f_{j}\eta \in \ha{\BB}(\lambda)_{\cl}$.

\paragraph{Case 4.} Assume that 
$x_{u} = \mcr{s_{j}x_{u+1}}$ and $t_{1}=\sigma_{m+1}$.
Then we deduce from the definition of the root operator $f_{j}$ 
(for the case $j=0$, see also Remark~\ref{rem:ro_f0}; 
 cf. \cite[Proposition~4.7\,a) and Remark~4.8]{Lit-A}) that
\begin{align*}
& f_{j}\eta=
(x_{1},\,x_{2},\,\dots,\,x_{u}=\mcr{s_{j}x_{u+1}},\,
 \mcr{s_{j}x_{u+2}},\,\dots,\, \\
& \hspace*{50mm} 
 \mcr{s_{j}x_{m}},\,\mcr{s_{j}x_{m+1}},\,x_{m+2},\,\dots,\,x_{s}\,;\, \\
& \hspace*{50mm}
\sigma_{0},\,\sigma_{1},\,\dots,\,\sigma_{u-1},\,\sigma_{u+1},\,\dots,\,
\sigma_{m},\,t_{1}=\sigma_{m+1},\,\dots,\,\sigma_{s});
\end{align*}
note that $\mcr{s_{j}x_{m+1}} \ne x_{m+2}$ since 
$\pair{s_{j}x_{m+1}\Lambda}{\alpha_{j}^{\vee}} < 0$ and 
$\pair{x_{m+2}\Lambda}{\alpha_{j}^{\vee}} \ge 0$ (see Case 2 above). 
In order to prove that $f_{j}\eta \in \ti{\BB}(\lambda)_{\cl}$, 
we need to verify that 
\begin{enumerate}
\renewcommand{\labelenumi}{(\roman{enumi})}
\setcounter{enumi}{1}

\item there exists a directed $\sigma_{p}$-path from 
$\mcr{s_{j}x_{p+1}}$ to $\mcr{s_{j}x_{p}}$ 
for each $u+1 \le p \le m$; 

\setcounter{enumi}{3}

\item there exists a directed $\sigma_{m+1}$-path from 
$x_{m+2}$ to $\mcr{s_{j}x_{m+1}}$ 
(when $m+1 < s$). 

\end{enumerate}
We can verify (ii) and (iv) 
by the same argument as for (ii) in Case 1 and (iv) in Case 2, respectively. 
Also, as in the proofs of (ii) in Case 1 and (iv) in Case 2, 
we can check that if $\eta \in \ha{\BB}(\lambda)_{\cl}$, then 
the directed paths in (ii) and (iv) above can be chosen from the shortest ones. 
Thus we have proved that $f_{j}\eta \in \ti{\BB}(\lambda)_{\cl}$, and 
that $\eta \in \ha{\BB}(\lambda)_{\cl}$ implies 
$f_{j}\eta \in \ha{\BB}(\lambda)_{\cl}$.

This completes the proof of Proposition~\ref{prop:stable}. 
\end{proof}

Combining Theorem~\ref{thm:charls} with 
Propositions~\ref{prop:ip} and \ref{prop:stable}, 
we obtain Theorem~\ref{thm:main}.


{\small
\setlength{\baselineskip}{13pt}
\renewcommand{\refname}{References}

}

\end{document}